\documentclass{amsart}

\usepackage{latexsym,amsmath,amssymb,amsfonts,amscd,graphics, color}
\usepackage[mathscr]{eucal}
\usepackage{mathrsfs}
\usepackage{tikz-cd}
\usepackage{hyperref}
\hypersetup{
	colorlinks=true,
	linkcolor=blue,
	filecolor=blue,      
	urlcolor=cyan,
	citecolor=magenta}

 \newif\ifHideFoot
\HideFoottrue   
\HideFootfalse   
\ifHideFoot

\newcommand{\Yano}[1]{}
\newcommand{\Lisa}[1]{}
\newcommand{\Zheng}[1]{}

\else

\newcommand{\marg}[1]{\normalsize{{
   \color{red}\footnote{{\color{blue}#1}}}{\marginpar[\vskip
   -.25cm{\color{red}\hfill\thefootnote$\implies$}]{\vskip
    -.2cm{\color{red}$\impliedby$\tiny\thefootnote}}}}}
\newcommand{\Yano}[1]{\marg{(Yano) #1}}
\newcommand{\Lisa}[1]{\marg{(Lisa) #1}}
\newcommand{\Zheng}[1]{\marg{(Zheng) #1}}

\fi

\numberwithin{equation}{section}
\theoremstyle{plain}
\newtheorem{theorem}[equation]{Theorem}

\newtheorem{lemma}[equation]{Lemma}
\newtheorem{proposition}[equation]{Proposition}

\theoremstyle{remark}
\newtheorem{remark}[equation]{Remark}

\theoremstyle{definition}
\newtheorem{definition}[equation]{Definition}

\newcommand{\bP}{\mathbb{P}}

\newcommand{\bQ}{\mathbb{Q}}
\newcommand{\bZ}{\mathbb{Z}}

\newcommand{\bC}{\mathbb{C}}

\newcommand{\bL}{\mathbb{L}}

\newcommand{\calA}{\mathcal{A}}

\newcommand{\calH}{\mathcal{H}}

\newcommand{\calR}{\mathcal{R}}

\newcommand{\calM}{\mathcal{M}}
\newcommand{\calO}{\mathcal{O}}
\newcommand{\calL}{\mathcal{L}}
\newcommand{\calP}{\mathcal{P}}

\newcommand{\Aut}{\mathrm{Aut}}

\newcommand{\Sym}{\mathrm{Sym}}

\newcommand{\Hom}{\mathrm{Hom}}

\newcommand{\Spec}{\mathrm{Spec}}

\newcommand{\Proj}{\mathrm{Proj}}

\newcommand{\im}{\mathrm{Im}}

\newcommand{\Pic}{\mathrm{Pic}}

\newcommand{\Lin}{\mathrm{Lin}}

\newcommand{\git}{/\kern-0.2em/}

\title[]{The moduli space of cubic threefolds with a non-Eckardt type involution via intermediate Jacobians} 

\author[]{Sebastian Casalaina-Martin}
\address{Department of Mathematics, University of Colorado, Boulder, CO 80309, USA}
\email{casa@math.colorado.edu}
\author[]{Lisa Marquand}
\address{Department of Mathematics, Stony Brook University,  Stony Brook, NY 11794, USA}
\email{lisa.marquand@stonybrook.edu}
\author[]{Zheng Zhang}
\address{Institute of Mathematical Sciences, ShanghaiTech University, Shanghai 201210, China}
\email{zhangzheng@shanghaitech.edu.cn}
\thanks{Research of the first named author is supported in part by a grant from the Simons Foundation (581058). Research of the second named author is supported in part by NSF grant DMS-2101640 (PI Laza). Research of the third named author is supported in part by NSFC grant 12201406.}
\date{\today}
\begin{document}
\bibliographystyle{halpha}
\maketitle

\begin{abstract}
	There are two types of involutions on a cubic threefold: the Eckardt type (which has been studied by the first named and the third named authors) and the non-Eckardt type. Here we study cubic threefolds with a non-Eckardt type involution, whose fixed locus consists of a line and a cubic curve. Specifically, we consider the period map sending a cubic threefold with a non-Eckardt type involution to the invariant part of the intermediate Jacobian. The main result is that the global Torelli Theorem holds for the period map. To prove the theorem, we project the cubic threefold from the pointwise fixed line and exhibit the invariant part of the intermediate Jacobian as a Prym variety of a (pseudo-)double cover of stable curves. The proof relies on a result of Ikeda and Naranjo-Ortega on the injectivity of the related Prym map. We also describe the invariant part of the intermediate Jacobian via the projection from a general invariant line and show that the two descriptions are related by the bigonal construction. 
\end{abstract}

\section*{Introduction} 
	
	Moduli spaces of cubic hypersurfaces are a central object of moduli theory, as they are one of the first examples one can study via a Hodge theoretic period map. Clemens, Griffiths \cite{MR302652}, Mumford \cite{MR0379510} and Beauville \cite{MR672617} proved the global Torelli Theorem for cubic threefolds -- namely, a cubic threefold is determined up to isomorphism by its intermediate Jacobian. Based on the work of Voisin, Hassett, Laza and Looijenga on the period map for cubic fourfolds, Allcock, Carlson and Toledo \cite{MR2789835} and Looijenga and Swierstra \cite{MR2339838} have exhibited the moduli space of cubic threefolds as a ball quotient. More specifically, this is achieved via the (eigen)period map for cubic fourfolds admitting an automorphism of order $3$, which are obtained as triple covers of $\bP^4$ branched along a cubic threefold. Furthermore, Kudla and Rapoport \cite{MR3001805} (see also \cite{MR4298656}) have interpreted the above construction as a certain map of stacks taking values in a moduli stack of abelian varieties of Picard type; in this way they are also able to describe the field of definition of the period map. It is worth noting that cubic hypersurfaces with additional automorphisms are related to other interesting moduli problems, and have been key ingredients for constructing several new period maps. Besides the moduli of cubic threefolds \cite{MR2789835, MR2339838}, examples include moduli of cubic surfaces \cite{MR1910264} (via cubic threefolds with an order $3$ automorphism), moduli of  cubic threefold pairs \cite{MR3886178} and cubic surface pairs \cite{MR4313238} (via cubic fourfolds and cubic threefolds admitting an Eckardt type involution, respectively). In a different direction, cubic threefolds with extra symmetry provide examples of unlikely intersections in the intermediate Jacobian locus \cite{MR4313238}. 
	
	Involutions on cubic fourfolds have been recently studied in \cite{MR3886178}, \cite{MR4190414} and \cite{marquandcubic} (see also \cite{MR4363785}). In this paper, we focus on cubic threefolds admitting a (biregular) involution -- these have been classified in for instance \cite{MR2820585}. In particular, there are two types of involutions for a cubic threefold; admitting one type is equivalent to having an Eckardt point. The moduli space of cubic threefolds admitting an Eckardt type involution has been studied in \cite{MR4313238}; the main result is that the period map sending an Eckardt cubic threefold to the anti-invariant part of the intermediate Jacobian is injective. The purpose of this paper is to study the analogous situation for the remaining involution.
 
 	More concretely, we study the moduli space $\calM$ of cubic threefolds $X\subset \bP^4$ with an involution $\tau$ of non-Eckardt type, whose fixed locus in $X$ consists of a line $L$ and a cubic curve $C$. We define $JX^\tau$ to be the invariant part of the intermediate Jacobian $JX$ with respect to the induced involution $\tau.$ The abelian subvariety $JX^\tau\subset JX$ is of dimension $3$ and inherits a polarization of type $(1,2,2)$, and thus we obtain a period map:
	\[\calP:\calM \longrightarrow \calA_3^{(1,2,2)}\]
	\[(X,\tau) \mapsto JX^\tau.\]

	Our main result is the following global Torelli theorem for $\calP$.
\begin{theorem}[Global Torelli for cubic threefolds with a non-Eckardt type involution; Theorem \ref{global torelli}]
\label{main}
	The period map $\calP:\calM \rightarrow \calA_3^{(1,2,2)}$, which sends a cubic threefold $X$ with a non-Eckardt type involution $\tau$ to the invariant part $JX^\tau\subset JX$, is injective. 
\end{theorem}
	
	We also prove that the infinitesimal Torelli theorem holds for $\calP:\calM \rightarrow \calA_3^{(1,2,2)}$ over an open subset $\calM_0\subset \calM$ (this is analogous to the situation for the moduli space of smooth curves of genus greater than $2$, where the infinitesimal Torelli theorem holds for the non-hyperelliptic locus; see~Remark \ref{M0hyperelliptic}).
	
\begin{proposition}[Infinitesimal Torelli for cubic threefolds with a non-Eckardt type involution; Proposition \ref{inftorelli}]
\label{maininftorelli}
	Let $\calM_0\subset \calM$ be the open subset described in \S\ref{inftorellinonEc}. The differential $d\calP$ of the period map $\calP$ is an isomorphism at every point of $\calM_0\subset \calM$. Combining this with Theorem \ref{main}, $\calP|_{\calM_0}:\calM_0\rightarrow \calA_3^{(1,2,2)}$ is an open embedding. 
\end{proposition}

	The strategy for proving the global Torelli theorem for the period map $\calP$ is similar to that in \cite{MR4313238}. Specifically, we project $(X,\tau)$ from the pointwise fixed line $L\subset X$ to realize the invariant part $JX^\tau\subset JX$ as a Prym variety. In particular, we show that $JX^\tau$ is isomorphic to the dual abelian variety of the Prym variety $P(\widetilde{C},C)$ of a double cover $\pi:\widetilde{C}\rightarrow C$ of a genus $1$ curve $C$ branched in six points (see Theorem \ref{inv/antiinv part1}). The crucial element in the proof is the description of Prym varieties for (pseudo-)double covers of stable curves given in \cite{MR472843}. The associated Prym map (recall that $\calR_{g,2n}$ is the moduli space of double covers of smooth genus $g$ curves branched in $2n$ distinct points)
	\[\calP_{1,6}:\mathcal{R}_{1,6}\rightarrow \calA_3^{(1,1,2)}\] 
is known to be injective (cf.~ \cite{MR4156417} or \cite{NO_torelli}), allowing us to recover $\pi:\widetilde{C}\rightarrow C$ from $JX^\tau$. We then apply the reconstruction result in \cite{MR1786479} (see also \cite{MR2123232}) to prove Theorem \ref{main}, noting that the line bundle $\calL$ associated with the double cover $\pi$ allows one to embed $C$ into $\bP^2$ as a plane cubic (cf.~Proposition \ref{thetachar}).

	We also study the fibration in conics obtained via the projection of a (general) cubic threefold $X$ with a non-Eckardt type involution $\tau$ from a (general) invariant line $l\subset X$. Invariant but not pointwise fixed lines in $X$ are parameterized by the curve $\widetilde{C}$. In this direction, we prove that the invariant part $JX^\tau$ is isomorphic to the Prym variety $P(D_\tau,\overline{D}_l)$ associated with a double cover $b_\tau:D_\tau\rightarrow \overline{D}_l$ of a genus $2$ curve $\overline{D}_l$ ramified in four points (c.f.~Theorem \ref{inv/antiinv part2}). The main techniques used in the proof are those developed in \cite{MR0379510}, \cite{MR1188194} and \cite{MR2218008} for studying Galois covers of curves with automorphism group the Klein four group. Letting the invariant line $l$ vary, one would expect the generic injectivity of the natural map from $\widetilde{C}$ to the generic fiber of the associated Prym map 
	\[\calP_{2,4}:\mathcal{R}_{2,4}\rightarrow \calA_3^{(1,2,2)}.\] 
However, this is not the case -- if two invariant lines $l$ and $l'$ form a coplanar pair (corresponding to a point $c_{l\cup l'}\in C$) meeting the pointwise invariant line $L$, then the double covers $b_\tau:D_\tau\rightarrow \overline{D}_l$ and $b'_\tau:D'_\tau\rightarrow \overline{D}_{l'}$ are isomorphic. Indeed, we prove the following result which relates the double covers $b_\tau:D_\tau\rightarrow \overline{D}_l$ (respectively, $b'_\tau:D'_\tau\rightarrow \overline{D}_{l'}$) and $\pi:\widetilde{C}\rightarrow C$ via the bigonal construction (see for example \cite{MR1188194}). This allows us to apply the argument in \cite{NO_torelli} or \cite{MR4435960} to show that the generic fiber of $\calP_{2,4}:\mathcal{R}_{2,4}\rightarrow \calA_3^{(1,2,2)}$ over a general member $JX^\tau\in \calA_3^{(1,2,2)}$ is birational to the elliptic curve $C$ (cf.~Proposition \ref{genericfibC}). 
\begin{proposition}[Projection from the pointwise fixed line vs.~ projection from a general invariant line; Proposition \ref{bigonalLl}] \label{mainbigonalLl}
	Notation as above. The towers of double covers 
	\[D_\tau\stackrel{b_\tau}{\rightarrow} \overline{D}_l\stackrel{r}{\rightarrow} \bP^1;\,\,\,D'_\tau\stackrel{b'_\tau}{\rightarrow} \overline{D}_{l'}\stackrel{r'}{\rightarrow} \bP^1\]
are both bigonally related to the tower of double covers $\widetilde{C}\stackrel{\pi}{\rightarrow} C\stackrel{p}{\rightarrow}\bP^1$, 
where $r:\overline{D}_l\rightarrow \bP^1$ (respectively, $r':\overline{D}_{l'}\rightarrow \bP^1$) denotes the map determined by the unique $g^1_2$ of the genus $2$ curve $\overline{D}_l$ (respectively, $\overline{D}_{l'}$) and $p:C\rightarrow \bP^1$ is the projection map from the point $c_{l\cup l'}\in C$. In particular, $b_\tau:D_\tau\rightarrow \overline{D}_l$ and $b'_\tau:D'_\tau\rightarrow \overline{D}_{l'}$ are isomorphic.
 \end{proposition}
  
	Finally, we note that exhibiting the invariant part $JX^{\tau}$ as a Prym variety turns out to be crucial for the ongoing project of the second named author, in applying the LSV construction (cf.~ \cite{MR3710794} and \cite{sac2021birational}) to cubic fourfolds with a non-Eckardt type involution. This  is important in the geometric study of involutions of hyper-K\"ahler manifolds of $OG10$ type (see \cite{marquandOG10}), particularly those involutions induced from a cubic fourfold. Recall that the Prym construction of the intermediate Jacobian of a cubic threefold is central to the work in \cite{MR3710794}, which associates to a cubic fourfold a hyper-K\"ahler manifold of $OG10$ type. 
	
	We now give an outline of the paper; we work throughout over the complex numbers $\bC$. In \S\ref{prelims}, we introduce our objects of interest, namely cubic threefolds $X$ with a non-Eckardt type involution $\tau$. We also investigate lines that are invariant under such an involution. In \S\ref{conic fib 1}, we exhibit such a cubic threefold $X$ as a conic fibration via projection from the pointwise fixed line $L\subset X$ and describe the invariant part $JX^{\tau}$ as a Prym variety. Using this description, we prove global and infinitesimal Torelli theorems for the period map $\calP:\calM \rightarrow \calA_3^{(1,2,2)}$ in \S\ref{torelli}.  Finally, we discuss an alternative description for $J(X)^\tau$, obtained by projecting $X$ from an invariant line different from $L$ in \S\ref{conic fib 2}.
		
	\subsection*{Acknowledgements} We would like to thank Radu Laza, Angela Ortega, Gregory Pearlstein and Roy Smith for helpful conversations related to the subject. We are grateful to Michael Rapoport for the interest in the paper. Special thanks go to Giulia Sacc\`a; in particular, the proof of Proposition \ref{prymdecomp1} is similar to that of \cite[Prop.~3.10]{MR4313238} which was due to her. We also thank the referee for helpful suggestions. Research of the first named author is supported in part by a grant from the Simons Foundation (581058). Research of the second named author is supported in part by NSF grant DMS-2101640 (PI Laza). Research of the third named author is supported in part by NSFC grant 12201406.

\section{Cubic threefolds with a non-Eckardt type involution} \label{prelims}
	
	In \S\ref{geometry}, we introduce cubic threefolds $X$ with a non-Eckardt type involution $\tau$. We then define a period map $\calP$ for these cubic threefolds in \S\ref{period}. In order to study the period map $\calP$, we will project $X$ from an invariant line to exhibit the intermediate Jacobian $JX$ as a Prym variety; we investigate the $\tau$-invariant lines that are contained in $X$ in \S\ref{inv}.

	\subsection{Involutions of cubic threefolds} \label{geometry} Let $X\subset \bP^{n+1}$ be a smooth hypersurface of degree $d$. Denote by $\Aut(X)$ the group of automorphisms of $X$, and by $\Lin(X)$ the subgroup of $\Aut(X)$ consisting of automorphisms induced by projective transformations of the ambient projective space leaving $X$ invariant. By \cite[Thm.~ 1 and 2]{MR168559}, assuming $n\geq2, d\geq 3$ we have that $\Aut(X)=\Lin(X)$, except in the case $n=2, d=4$. Moreover, $\Aut(X)$ is finite (again excluding the case $n=2, d=4$). As a consequence (and specifying to the case $n=d=3$), one can obtain a complete classification of prime order automorphisms of smooth cubic threefolds (e.g.~\cite[Thm.~ 3.5]{MR2820585}, see also the references in \cite[Rmk.~1.6]{MR4363785}). In particular, for involutions we have the following classification.
\begin{proposition} \label{tau1tau2}
	Let $X=V(F)$ be a smooth cubic threefold in $\bP^4$ that admits an involution $\tau$. Applying a linear change of coordinates, we can diagonalize $\tau$, so that 
	\[\tau:\bP^4\rightarrow \bP^4, \,\,\, [x_0,\dots, x_4]\mapsto [(-1)^{a_0}x_0,\dots, (-1)^{a_4}x_4],\] with $a_i\in \{0,1\}$. Let $a:=(a_0,\dots a_4)$, and let $D$ be the dimension of the family of smooth cubic threefolds that admit the involution $\tau$. Then
\begin{enumerate}
	\item either $a=(0,0,0,0,1)$ and $\tau=\tau_1$ fixes pointwise a hyperplane section $S\subset X$ and a point $p\in X\setminus S$. We have that $D=7$ and 
	\[F=f(x_0,x_1,x_2,x_3)+\ell(x_0,x_1,x_2,x_3)x_4^2,\]
		where $\ell$ is a homogeneous linear polynomial, and $f$ is homogeneous of degree $3$.
	\item or $a=(0,0,0,1,1)$ and $\tau=\tau_2$ fixes pointwise a line $l\subset X$ and a plane cubic $C\subset X$. We have that $D=6$ and
	\[F= x_0q_0(x_3, x_4)+x_1q_1(x_3, x_4)+x_2q_2(x_3,x_4)+g(x_0,x_1,x_2),\]
		 where each $q_i$ is a homogeneous quadratic polynomial and $g$ is a homogeneous cubic polynomial.
\end{enumerate}
\end{proposition}
\begin{proof}
	See \cite[Thm.~ 3.5]{MR2820585}.
\end{proof}
	
Admitting an involution of type $\tau_1$ is equivalent to the existence of an Eckardt point. Such a cubic is called an \emph{Eckardt} cubic and has been well studied (see for example \cite{MR3886178} and \cite{MR4313238}). In this paper, we will focus on studying involutions of the type $\tau_2$; we make the following definition.

\begin{definition} \label{non-eckardt}
	We call an involution on a smooth cubic threefold of type $\tau_2$ (as in Proposition \ref{tau1tau2}) an involution of \emph{non-Eckardt type}.
\end{definition}

Throughout, $X\subset \bP^4$ is a smooth cubic threefold with an involution $\tau$ of non-Eckardt type with equation
\begin{equation} \label{eqn:X}	
	F= x_0q_0(x_3, x_4)+x_1q_1(x_3, x_4)+x_2q_2(x_3,x_4)+g(x_0,x_1,x_2)=0,
\end{equation} 
where each $q_i(x_3,x_4)$ is homogeneous of degree $2$ and $g(x_0,x_1,x_2)$ is homogeneous of degree $3$. The involution 
\begin{equation} \label{eqn:tau}
\tau:[x_0,x_1,x_2,x_3,x_4]\mapsto[x_0,x_1,x_2,-x_3,-x_4]
\end{equation}
fixes two complementary linear subspaces of $\bP^4$ pointwise; the line $L:=V(x_0,x_1,x_2)$ and the plane $\Pi:=V(x_3,x_4)$. Notice that the line $L\subset X$, and the fixed curve $C$ is given by the intersection $X\cap \Pi$; i.e. $C=V(g(x_0,x_1,x_2), x_3, x_4)$.

\begin{lemma} \label{Csmooth}
	Let $(X,\tau)$ be a smooth cubic threefold with a non-Eckardt type involution. Then the fixed curve $C \subset X$ as above is smooth.
\end{lemma}
\begin{proof}
	Suppose that $C$ is not smooth. Then there exists $a=[a_0,a_1,a_2]\in C\subset \Pi\cong \bP^2_{x_0,x_1,x_2}$ such that $\frac{\partial g}{\partial x_i}(a)=0$ for $i=0,1, 2.$ Taking partial derivatives of the Equation (\ref{eqn:X}) shows that $X$ is singular at the point $[a_0,a_1,a_2,0,0]$. 
\end{proof}

	\subsection{The period map for cubic threefolds with a non-Eckardt type involution} \label{period} Let $X$ be a cubic threefold with the involution $\tau$ of non-Eckardt type as discussed in the previous subsection. By abuse of notation, we use $\tau$ to denote the involution on the principally polarized intermediate Jacobian $JX$ induced by the involution $\tau$ of $X$. Define the invariant part $JX^\tau$ and the anti-invariant part $JX^{-\tau}$ respectively by 
\begin{equation}
	JX^\tau := \im(1+\tau);\,\,\, JX^{-\tau}:=\im(1-\tau).
\end{equation}
Note that $JX^\tau$ and $JX^{-\tau}$ are $\tau$-stable complementary abelian subvarieties of $JX$ (cf.~\cite[Prop.~ 13.6.1]{MR2062673}). 

\begin{lemma} \label{JXsplit}
	The abelian subvarieties $JX^\tau$ and $JX^{-\tau}$ have dimensions 3 and 2 respectively. The principal polarization of $JX$ induces polarizations of type $(1,2,2)$ and $(2,2)$ on $JX^\tau$ and $JX^{-\tau}$ respectively.
\end{lemma}
\begin{proof}
	The abelian subvarieties $JX^\tau$ and $JX^{-\tau}$ correspond to the symmetric idempotents $\frac{1+\tau}{2}$ and   
$\frac{1-\tau}{2}$ in $\mathrm{End}_\bQ(JX)$ respectively. Using \cite[Prop.~ 5.3.10]{MR2062673}, we compute their dimensions by studying the eigenspace decomposition of $\tau$ on $H^{1,2}(X)$ or $H^{2,1}(X)$ for a particular smooth cubic threefold with an  involution of non-Eckardt type (e.g.~ $V(x_0x_3^2+x_0x_4^2+x_1x_3^2+x_2x_4^2-x_0^3+x_1^3+x_2^3)$). Identifying the eigenspaces is a standard computation using Griffiths residues (see for instance \cite[Thm.~ 3.2.10]{MR3727160}). The claim on the polarization types will be proved later in Theorem \ref{inv/antiinv part1} (see also Theorem \ref{inv/antiinv part2}). Note that here the number of $2$'s in the polarization types for $JX^\tau$ and $JX^{-\tau}$ are the same which for instance follows from \cite[Lem.~ 1.13]{MR4313238}.
\end{proof}

	Let $\calM$ be the moduli space of cubic threefolds $X$ with an involution $\tau$ of non-Eckardt type constructed using GIT (see for example \cite[\S2.2]{MR4190414}). Let $\calA_3^{(1,2,2)}$ be the moduli space of abelian threefolds with a polarization of type $(1,2,2)$. Note that $\dim \calM=6$ and $\dim\calA_3^{(1,2,2)}=6.$ Define a period map (via Lemma \ref{JXsplit}):	
	\[\calP:\calM \longrightarrow \calA_3^{(1,2,2)}\]
	\[(X,\tau) \mapsto JX^\tau\]
which sends a smooth cubic threefold $X$ with an involution $\tau$ of non-Eckardt type to the invariant part $JX^\tau$ of the intermediate Jacobian $JX$.

	\subsection{Invariant lines} \label{inv} In order to study the period map $\calP$ in \S\ref{period}, we will need to understand how the intermediate Jacobian $JX$ of such a cubic threefold $(X,\tau)$ decomposes with respect to the involution $\tau$. As in \cite{MR4313238}, our strategy will be to project $X$ from a $\tau$-invariant line to exhibit $JX$ as the Prym variety of the associated discriminant double cover.

\begin{lemma} \label{inv lines}
	Let $X$ be a smooth cubic threefold with an involution $\tau$ of non-Eckardt type, cut out by Equation \eqref{eqn:X}. Let $l\subset X$ be a $\tau$-invariant line. Then either $l$ is pointwise fixed by $\tau$ (i.e. $l=L=V(x_0,x_1,x_2)$), or $l$ intersects both the fixed line $L$ and the fixed curve $C=V(g(x_0,x_1,x_2), x_3, x_4) \subset X$.
\end{lemma}
\begin{proof}
	We use the notation in \S\ref{geometry}. The fixed locus of the involution $\tau$ acting on $\bP^4$ consists of the line $L$ and the plane $\Pi=V(x_3,x_4)$. If $\tau$ fixes every point of $l\subset C$, then either $l=L$, or $l$ is a component of $C$. By Lemma \ref{Csmooth}, the curve $C$ is smooth, and so $l=L$. 
	Otherwise, $\tau$ fixes two points of $l\subset X$. One of the points needs to be off of the fixed line $L$, and hence must be a point of $C$. Thus $l$ intersects both $L$ and $C$.
\end{proof}
 
 	Observe that the plane $\langle L,l\rangle$ spanned by the pointwise fixed line $L$ and an invariant line $l\neq L$ is itself $\tau$-invariant and therefore must intersect $X$ along a third invariant line $l'$. Through projecting $X$ from the pointwise fixed line $L\subset X$, we will see in \S\ref{projX1} (and also the proof of Proposition \ref{prymdecomp1}) that the $\tau$-invariant lines $l\subset X$ which are not pointwise fixed are parameterized by a smooth genus $4$ curve $\widetilde{C}$ which is a double cover of $C$. In other words, the fixed locus of $\tau$ on the Fano surface $F(X)$ of lines consists of a point corresponding to the pointwise fixed line $L$ and the other curve $\widetilde{C}$ parameterizing other $\tau$-invariant lines $l$: $F(X)^\tau = \{L\} \cup \widetilde{C}$.

\section{Cubic threefolds with an involution of non-Eckardt type as fibrations in conics I: pointwise fixed line} \label{conic fib 1}
	In this section, we study the intermediate Jacobians $JX$ of cubic threefolds $(X,\tau)$ with an involution of non-Eckardt type via projections from the pointwise fixed lines $L\subset X$. Some basic facts about cubic threefolds as fibrations in conics are first recalled in \S\ref{general conic fib}. We then focus on cubic threefolds with a non-Eckardt type involution and study the fibrations in conics obtained by projecting these cubic threefolds from pointwise fixed lines in \S\ref{projX1}. An important observation is that the discriminant quintic curves split as the union of smooth cubic curves and transverse quadratic curves. Based on the observation and the results in \cite[\S0.3]{MR472843}, we give a characterization of the invariant and anti-invariant parts of the intermediate Jacobians in \S\ref{IJ as prym 1}.
	
	\subsection{Cubic threefolds as fibrations in conics} \label{general conic fib} Let $X\subset \bP^4$  be a smooth cubic threefold with a line $l\subset X$. The linear projection with center $l$ expresses $X$ as a conic fibration over a complementary plane $\bP^2$; indeed, $\bP^2$ also parametrizes the space of $\bP^2$-sections of $X$ containing $l$. The blow up $\mathrm{Bl}_l\bP^4$ of the ambient projective space along $l$ gives a commutative diagram
\[ \begin{tikzcd}
	\mathrm{Bl}_lX \arrow[hook]{r} \arrow[swap]{dr}{\pi_l} & \mathrm{Bl}_l\bP^4 \arrow{d} \\%
	& \bP^2
\end{tikzcd}\]
where $\mathrm{Bl}_lX$ is the strict transform of $X$ in $\mathrm{Bl}_l\bP^4$, yielding a fibration in conics $\pi_l:\mathrm{Bl}_lX \rightarrow \bP^2$. The discriminant curve is a plane quintic $D\subset \bP^2$ which by \cite[Prop.~ 1.2]{MR472843} is stable, and there is an associated pseudo-double cover\footnote{Let $\widetilde{D}\rightarrow D$ be a double cover of stable curves with the associated covering involution $\iota$. We say that $\widetilde{D}\rightarrow D$ is \emph{admissible} if the fixed points of $\iota$ are nodes, and the local branches are not interchanged by $\iota$ at each fixed node of $\widetilde{D}$. An admissible double cover $\widetilde{D}\rightarrow D$ is called \emph{allowable} (see \cite[(**), p.~173]{MR572974} and \cite[\S I.1.3]{MR594627}) if the associated Prym is compact. An allowable double cover $\widetilde{D}\rightarrow D$ is said to be a \emph{pseudo-double cover} (cf.~ \cite[(*), p.~157]{MR572974} and \cite[Def.~ 0.3.1]{MR472843}) if the fixed points of $\iota$ are exactly the nodes of $\widetilde{D}$.} $\pi:\widetilde{D}\rightarrow D$ determined by interchanging the lines in the fiber of $\pi_l$ over the points of $D$ (cf.~\cite[Prop.~ 1.5]{MR472843}). For a projection from a general line $l \subset X$, $D$ is smooth and $\widetilde{D}\rightarrow D$ is connected and \'etale (see for example \cite[Appendix C]{MR302652}). 

	Associated with the discriminant double cover $\pi:\widetilde{D}\rightarrow D$ is a rank-$1$ torsion-free sheaf $\eta_D$ which is reflexive, i.e.~ $\calH om(\eta_D, \calO_D) \cong \eta_D$ (more precisely, $\widetilde{D}\rightarrow D$ is constructed as $\underline{\Spec}_D(\calO_D\oplus \eta_D) \rightarrow D$ where the $\calO_D$-algebra structure on $\calO_D\oplus \eta_D$ is induced by $\calH om(\eta_D, \calO_D) \cong \eta_D$). Let $\kappa_D:=\eta_D\otimes \calO_D(1)$. Then $\kappa_D$ is an odd theta characteristic satisfying $\calH om(\kappa_D, \omega_D) \cong \kappa_D$ and $h^0(D, \kappa_D)=1$. Note also that $\kappa_D$ (and therefore $\eta_D$) is locally free at a point $d\in D$ if and only if $\widetilde{D}\rightarrow D$ is \'etale over $d$. By \cite[\S1.6]{MR472843} and \cite[Prop.~ 4.2]{MR2123232}, the conic fibration construction gives a one-to-one correspondence between pairs $(X, l)$ consisting of a smooth cubic threefold $X$ and a line $l \subset X$ and pairs $(D, \kappa_D)$ where $D$ is a stable plane quintic curve and $\kappa_D$ is a theta characteristic with $h^0(D, \kappa_D)=1$, both up to projective linear transformations. 

The above construction can also be described in coordinates. We may assume that $l \subset \bP^4$ is cut out by $x_0=x_1=x_2=0$. Since $l \subset X$, the equation of $X$ is of the form
\begin{align*}
	&\ell_1(x_0,x_1,x_2)x_3^2+\ell_2(x_0,x_1,x_2)x_4^2+2\ell_3(x_0,x_1,x_2)x_3x_4 \\
	&+2q_1(x_0,x_1,x_2)x_3+2q_2(x_0,x_1,x_2)x_4+c(x_0,x_1,x_2)=0
\end{align*}
where $\ell_i$, $q_j$ and $c$ are homogeneous polynomials of degree $1$, $2$ and $3$ respectively. 
Let $M$ be the matrix 
\begin{equation}\label{E:Mdef}
M=\left(\begin{matrix}
	\ell_1(x_0,x_1,x_2) & \ell_3(x_0,x_1,x_2) & q_1(x_0,x_1,x_2) \\
	\ell_3(x_0,x_1,x_2) & \ell_2(x_0,x_1,x_2) & q_2(x_0,x_1,x_2) \\
	q_1(x_0,x_1,x_2) & q_2(x_0,x_1,x_2) &  c(x_0,x_1,x_2)
\end{matrix}\right).
\end{equation}
Then the discriminant quintic curve $D \subset V(x_3,x_4)\cong \bP^2$ for the conic fibration $\pi_l:\mathrm{Bl}_lX\rightarrow \bP^2$ is cut out by the discriminant of $M$: $D=V(\det(M))$. In particular, a point $d\in D$ is a smooth point if and only if the corank of $M$ at $d$ is $1$ (note that because $X$ is smooth the corank of $M$ is at most $2$). Moreover, following \cite[Prop.~ 4.2]{MR1786479} and \cite[Thm.~ 4.1]{MR2123232} the theta characteristic $\kappa_D$ admits a short exact sequence
	\[0\rightarrow \calO_{\bP^2}(-2)^{\oplus2}\oplus \calO_{\bP^2}(-3)\stackrel{M}{\rightarrow} \calO_{\bP^2}(-1)^{\oplus2}\oplus\calO_{\bP^2}\rightarrow \kappa_D\rightarrow 0.\]	
(Indeed, the smooth cubic threefold $X$ and the line $l \subset X$ can be recovered from the above minimal resolution of $\kappa_D$ up to projective linear transformations, cf.~\cite[Prop.~ 4.2]{MR2123232}.) When the plane discriminant quintic $D$ is smooth and the discriminant double cover $\pi:\widetilde{D}\rightarrow D$ is connected and \'etale, the theta characteristic $\kappa_D$ corresponds to the divisor $\sqrt{(\ell_1\ell_2-\ell_3^2=0)}$ (which is the unique effective divisor such that twice of the divisor is the divisor $(\ell_1\ell_2-\ell_3^2=0)$ on $D$), and the \'etale double cover $\widetilde{D}\rightarrow D$ is associated with the $2$-torsion line bundle $\eta_D=\kappa_D(-1)$.

	Denote the Prym variety of the discriminant pseudo-double cover $\pi:\widetilde{D}\rightarrow D$ by	$P(\widetilde{D}, D)$ which is defined as
	\[P(\widetilde{D}, D):=(\ker(\mathrm{Nm}_\pi:J(\widetilde{D})\rightarrow J(D)))^0\]	
(cf.~\cite[\S3]{MR0379510} and \cite[\S3]{MR572974}). For later use, let us give an explicit description of $P(\widetilde{D},D)$ following \cite[\S0.3]{MR472843}. Set $\tilde{\nu}:N(\widetilde{D})\rightarrow \widetilde{D}$ (respectively, $\nu:N(D)\rightarrow D$) to be the normalization of $\widetilde{D}$ (respectively, $D$). Denote by $\pi':N(\widetilde{D})\rightarrow N(D)$ the induced double cover. By \cite[Prop.~ 3.5]{MR572974}, there exists an isogeny $\tilde{\nu}^*:P(\widetilde{D},D)\rightarrow P(N(\widetilde{D}),N(D))$. More precisely, denote by $\Theta_{J(N(\widetilde{D}))}$ the principal polarization on $J(N(\widetilde{D}))$ and consider the restriction of $\Theta_{J(N(\widetilde{D}))}$ to $P(N(\widetilde{D}),N(D))$:
	\[\Theta_{P(N(\widetilde{D}),N(D))}:=\Theta_{J(N(\widetilde{D}))}|_{P(N(\widetilde{D}),N(D))}.\] 
By \cite[Thm.~ 3.7]{MR572974}, $\Theta_{P(N(\widetilde{D}),N(D))}$	 induces twice of a principal polarization $\Xi$ on $P(\widetilde{D},D)$: 
	\[(\tilde{\nu}^*)^{-1}\Theta_{P(N(\widetilde{D}),N(D))}\equiv_{\mathrm{alg}} 2\Xi.\] 
In other words, the isogeny above is an isogeny of polarized abelian varieties: 
	\[\tilde{\nu}^*:(P(\widetilde{D},D),2\Xi)\rightarrow (P(N(\widetilde{D}),N(D)),\Theta_{P(N(\widetilde{D}),N(D))}).\] 
From \cite[Prop.~12.1.3]{MR2062673}, we deduce that 
\[(P(N(\widetilde{D}),N(D))^\vee,(\Theta_{P(N(\widetilde{D}),N(D))})^\vee)\cong (J(N(\widetilde{D}))/\pi'^*J(N(D)),\Theta')\] where $\Theta'$ denotes the dual polarization\footnote{Here we are using the dual polarization defined in \cite[Thm.~2.1]{MR1976843} (see also \cite[Rmk.~1.14]{MR4313238}) which is slightly different from the one used in \cite[\S14.4]{MR2062673}. In particular, for a polarization of type $(d_1,d_2,\dots,d_g)$ the dual polarization has type $(\frac{d_g}{d_g},\frac{d_g}{d_{g-1}},\dots,\frac{d_g}{d_1})$.}. As a result, we get the dual isogeny of polarized abelian varieties
	\[(\tilde{\nu}^*)^\vee:(J(N(\widetilde{D}))/\pi'^*J(D),\Theta')\rightarrow (P(\widetilde{D},D)^\vee,(2\Xi)^\vee)\cong (P(\widetilde{D},D),\Xi).\]
For a pseudo-double cover, the kernel of $(\tilde{\nu}^*)^\vee$ has been described in \cite[\S0.3]{MR472843} (see also \cite[p.~76]{MR1944808}). Specifically, let $H'\subset \Pic(N(\widetilde{D}))$ be the subgroup generated by $\calO_{N(\widetilde{D})}(s-s')$ where $s,s'\in N(\widetilde{D})$ with $\tilde{\nu}(s)=\tilde{\nu}(s')$. Set $H$ to be the image of $H_0:=H'\cap J(N(\widetilde{D}))$ in the quotient $J(N(\widetilde{D}))/\pi'^*J(N(D))$. Then $H$ is the kernel of the isogeny of polarized abelian varieties
	\[0\rightarrow H\rightarrow J(N(\widetilde{D}))/\pi'^*J(D)\stackrel{(\tilde{\nu}^*)^\vee}{\rightarrow} P(\widetilde{D},D)\rightarrow 0.\]
	
	By \cite[Appendix C]{MR302652} and \cite[Thm.~ 2.1]{MR472843}, the conic fibration construction $\pi_l:\mathrm{Bl}_lX \rightarrow \bP^2$ induces a canonical isomorphism of principally polarized abelian varieties \[JX \cong P(\widetilde{D}, D)\] between the intermediate Jacobian $JX$ of the smooth cubic threefold $X$ and the Prym variety $P(\widetilde{D}, D)$ of the discriminant double cover $\widetilde{D}\rightarrow D$.

	\subsection{Projecting cubic threefolds with a non-Eckardt type involution from the pointwise fixed lines} \label{projX1} Let $(X,\tau)$ be a smooth cubic threefold with a non-Eckardt type involution as in \S\ref{geometry}. Let $L\subset X$ be the pointwise fixed line under $\tau$ (see Lemma \ref{inv lines}). We can rewrite the equation of $X$ in Equation \eqref{eqn:X} as 
\begin{equation} \label{rewrite eqn}
	\ell_1(x_0,x_1,x_2)x_3^2+\ell_2(x_0,x_1,x_2)x_4^2+2\ell_3(x_0,x_1,x_2)x_3x_4+g(x_0,x_1,x_2)=0,
\end{equation}
where $\ell_i(x_0, x_1, x_2)$ are linear forms, and $g(x_0, x_1, x_2)$ is homogeneous of degree $3$. As in \S\ref{general conic fib} we project $X$ from the fixed line $L=V(x_0, x_1, x_2)$ to the complementary plane $\Pi:=V(x_3, x_4) \cong \bP^2_{x_0,x_1,x_2}$, and obtain a fibration in conics $\pi_L:\mathrm{Bl}_LX\rightarrow \Pi$. 
The plane discriminant quintic $D_L\subset \Pi$ has equation $\det M=0$ where the matrix $M$ is 
\begin{equation} \label{M}
M=\left(\begin{matrix}
	\ell_1(x_0,x_1,x_2) & \ell_3(x_0,x_1,x_2) & 0 \\
	\ell_3(x_0,x_1,x_2) & \ell_2(x_0,x_1,x_2) & 0 \\
	0 & 0&  g(x_0,x_1,x_2)
\end{matrix}\right).
\end{equation}
We are ready to see that $D_L$ is nodal and consists of the following components: a smooth plane cubic 
	\[C:=V(g(x_0,x_1,x_2), x_3, x_4) \subset \Pi,\] 
and a (possibly degenerate) plane conic 
	\[Q:=V(\ell_1(x_0,x_1,x_2)\ell_2(x_0,x_1,x_2)-\ell_3^2(x_0,x_1,x_2), x_3, x_4) \subset \Pi.\]
Note that $C=X\cap \Pi$ is the pointwise fixed curve by the involution $\tau$ in \S\ref{geometry}.

\begin{lemma} \label{CQsmooth}
	Let $(X,\tau)$ be a smooth cubic threefold with a non-Eckardt type involution as above, and let $\pi_L:\mathrm{Bl}_LX\rightarrow \Pi$ be the projection from the fixed line $L\subset X$. Then the discriminant curve $D_L$ is a union of a cubic curve $C$ and a conic curve $Q$. Moreover, the cubic component $C$ is smooth, and $C$ meets the conic component $Q$ transversely.
\end{lemma}
\begin{proof}
	The claim that $C$ is smooth has been verified in Lemma \ref{Csmooth}.  It is shown in \cite[Prop.~1.2]{MR472843} that $D_L$ is at worst nodal, and therefore $C$ and $Q$ meet transversally. 
\end{proof}

	Denote by $\pi_L: \widetilde{D}_L  \rightarrow D_L=C\cup Q$ the discriminant double cover of the fibration in conics $\pi_L:\mathrm{Bl}_LX\rightarrow \Pi$. Note that $\pi_L:\widetilde{D}_L \rightarrow D_L$ is branched at the intersection points $C\cap Q$. Let us also observe that $\widetilde{D}_L=\widetilde{C}\cup \widetilde{Q}$ where $\widetilde{C}$ (respectively, $\widetilde{Q}$) is a double cover of the smooth cubic $C$ (respectively, the conic $Q$) ramified in the intersection points $C\cap Q$. We now restrict the discriminant double cover $\pi_L: \widetilde{D}_L=\widetilde{C}\cup \widetilde{Q}\rightarrow D_L=C\cup Q$ to $C$ and focus on the obtained double cover $\pi_L|_{\widetilde{C}}:\widetilde{C} \rightarrow C$ (if no confusion is likely to be caused we will simply write $\pi$ instead of $\pi_L|_{\widetilde{C}}$). Specifically, we describe the quadruple $(C,\beta,\calL,s)$ corresponding to $\pi: \widetilde{C}\rightarrow C$ where $\beta$ is the branch divisor, $\calL$ is the associated line bundle on $C$ satisfying $\calL^{\otimes2}\cong \calO_C(\beta)$ and $s$ is a section of $\calO_C(\beta)$ vanishing on $\beta$. For $\pi: \widetilde{C}\rightarrow C$, clearly one has $\beta=C\cap Q$ and $s=\ell_1(x_0,x_1,x_2)\ell_2(x_0,x_1,x_2)-\ell_3^2(x_0,x_1,x_2)$. It remains to determine the associated line bundle $\calL$.
	
\begin{proposition} \label{thetachar}
	The double cover $\pi: \widetilde{C}\rightarrow C$ obtained by restricting the discriminant double cover $\pi_L: \widetilde{D}_L \rightarrow D_L$ to the smooth cubic component $C$ is associated with the line bundle $\calL=\calO_C(1)$.
\end{proposition}
\begin{proof}
	We may assume that the conic $Q$ is smooth (note that this is the case for a general cubic threefold with a non-Eckardt type involution; a similar argument applies to the case when $Q$ has rank $2$). As in \S\ref{general conic fib}, set $\eta_{D_L}$ to be the rank-$1$ torsion-free sheaf which is associated with the discriminant double cover $\widetilde{D}_L \rightarrow D_L$ and satisfies $\calH om(\eta_{D_L}, \calO_D)\cong \eta_{D_L}$. Let $\kappa_{D_L}=\eta_{D_L}\otimes \calO_{D_L}(1)$ be the theta characteristic on $D_L$. By a result of Beauville (see \cite[Prop.~ 4.2]{MR1786479} and \cite[Thm.~ 4.1]{MR2123232}), there exists a short exact sequence 
	\[0\rightarrow \calO_{\bP^2}(-2)^{\oplus 2}\oplus \calO_{\bP^2}(-3)\stackrel{M}{\rightarrow} \calO_{\bP^2}(-1)^{\oplus 2}\oplus \calO_{\bP^2}\rightarrow \kappa_{D_L}\rightarrow 0\]
where $\bP^2$ denotes the plane $\Pi=V(x_3, x_4)\cong \bP^2_{x_0,x_1,x_2}$ and $M$ is the matrix in Equation \eqref{M}. Restricting the above exact sequence to the smooth cubic component $C$, one gets the following sequence which coincides with the closed subscheme sequence for $C=V(g)\subset \bP^2$:
	\[0\rightarrow \calO_{\bP^2}(-3)\stackrel{\cdot g}{\rightarrow} \calO_{\bP^2}\rightarrow (\kappa_{D_L}|_C)/\mathrm{torsion}\rightarrow 0.\]
As a consequence, we get $(\kappa_{D_L}|_C)/\mathrm{torsion}\cong \calO_C$. Similarly, from \cite[Prop.~ 4.2]{MR1786479} one deduces that $(\kappa_{D_L}|_Q)/\mathrm{torsion}$ is $\theta_Q:=\calO_Q(-pt)$. Since $\eta_{D_L}=\kappa_{D_L}\otimes \calO_{D_L}(-1)$, we have that $(\eta_{D_L}|_C)/\mathrm{torsion}\cong \calO_C(-1)$ and that $(\eta_{D_L}|_Q)/\mathrm{torsion}\cong \theta_Q(-1)$. Since $\pi_L: \widetilde{D}_L \rightarrow D_L$ is associated with the rank-$1$ reflexive sheaf $\eta_{D_L}$, it holds that $\pi_{L*}\calO_{\widetilde{D}_L}\cong \calO_{D_L}\oplus \eta_{D_L}$ (also compare \cite[Prop.~2.5]{MR4163532}). Pulling back to $C$, one gets 
	\[\pi_*\calO_{\widetilde{C}}\cong  (\pi_{L*}\calO_{\widetilde{D}_L}|_C)/\mathrm{torsion}\cong \calO_C\oplus (\eta_{D_L}|_C)/\mathrm{torsion}.\] 
The isomorphism on the left arises as follows. The pull-back and push-forward functors induce a natural morphism  $\pi_{L*}\calO_{\widetilde{D}_L}|_C \rightarrow \pi_*(\mathcal O_{\widetilde{D}_L}|_{\widetilde C})=\pi_*\mathcal O_{\widetilde C}$.  Since $\pi_*\mathcal O_{\widetilde C}$ is a rank-$2$ vector bundle, the morphism factors through the quotient by the torsion sub-sheaf, and then a local computation at the nodes shows the morphism is an isomorphism. Now, because $(\eta_{D_L}|_C)/\mathrm{torsion}\cong \calO_C(-1)$, the line bundle $\calL$ determining the double cover $\pi:\widetilde{C}\rightarrow C$ is then $((\eta_{D_L}|_C)/\mathrm{torsion})^\vee\cong \calO_C(1)$.
\end{proof}

\begin{remark}
	We give the following characterizations of the rank-$1$ torsion-free reflexive sheaf $\eta_{D_L}$. Without loss of generality, we assume the conic component $Q$ of the discriminant quintic $D_L$ is smooth.
\begin{enumerate}
	\item Let $\nu: C\coprod Q\rightarrow D_L=C\cup Q$ be the normalization map. Denote the preimages of the six intersection points $d_1,\dots,d_6 \in C\cap Q$ on $C$ (respectively, $Q$) by $c_1,\dots,c_6$ (respectively, $q_1,\dots,q_6$). Note that $\eta_{D_L}$ is not locally free at the intersection points $C\cap Q$. Pulling back $\eta_{D_L}$ via the normalization map $\nu$ gives a line bundle $\calL'$ (denote the corresponding geometric line bundle by $\bL'$) on $C\coprod Q$ together with gluing maps along fibers $\alpha_{CQ,i}=0:\bL'_{c_i}\rightarrow \bL'_{q_i}$ which are all zero for $1\leq i\leq 6$. Equivalently, we could also describe $\eta_{D_L}$ using the data $(\calL'':=\calL'(\Sigma_ic_i-\Sigma_iq_i),\alpha_{QC,i}=0:\bL''_{q_i}\rightarrow \bL''_{c_i})$ (again $\bL''$ denotes the geometric line bundle corresponding to $\calL''$). From the proof of Proposition \ref{thetachar}, we deduce that $\calL'|_C=(\eta_{D_L}|_C)/\mathrm{torsion}\cong \calO_C(-1)$ and $\calL''|_Q=(\eta_{D_L}|_Q)/\mathrm{torsion}\cong \theta_Q(-1)$. In other words, $\eta_{D_L}$ corresponds to the data $(\calL_C\cong \calO_C(-1),\calL_Q(\Sigma_iq_i)\cong \theta_Q(2),\alpha_{CQ,i}=0)$, or equivalently, to the data $(\calL_C(\Sigma_ic_i)\cong \calO_C(1),\calL_Q\cong \theta_Q(-1),\alpha_{QC,i}=0)$. In particular, it holds that $\eta_{D_L}\cong \calH om(\eta_{D_L}, \calO_D)$.
	\item Alternatively, we describe the rank-$1$ torsion-free sheaf $\eta_{D_L}$ using line bundles over a semistable model of $D_L=C\cup Q$. Specifically, set $T:=\underline{\Proj}_{D_L}\Sym^\bullet \eta_{D_L}$. Note that $\eta_{D_L}$ is not locally free at the nodes $C\cap Q$. Then $T$ is a semistable curve obtained by replacing every intersection point $d_i\in C\cap Q$ with $1\leq i\leq 6$ by a smooth rational component $E_i$ which meets the component $C$ at $c_i$ and the component $Q$ at $q_i$. Moreover, $T$ admits a natural map $\mathrm{st}:T\rightarrow D_L$ contracting the exceptional components $E_1,\dots,E_6$. Letting $\xi$ be the tautological invertible sheaf on $T$ which has total degree $0$, we have that $\mathrm{st}_*\xi=\eta_{D_L}$ and $\xi|_{E_i}\cong \calO_{E_i}(1)$. From $(\eta_{D_L}|_C)/\mathrm{torsion}\cong \calO_C(-1)$ and $(\eta_{D_L}|_Q)/\mathrm{torsion}\cong \theta_Q(-1)$, one gets that $\xi|_C\cong \calO_C(-1)$ and $\xi|_Q\cong \theta_Q(-1)$. Now set $\gamma:\xi^{\otimes 2}\rightarrow \calO_T$ to be the homomorphism which vanishes on the exceptional components $E_1,\dots,E_6$ and coincides with $(\xi|_C)^{\otimes 2} \cong \calO_C(-\Sigma_ic_i)\hookrightarrow \calO_C$ on $C$ and $(\xi|_Q)^{\otimes 2} \cong \calO_Q(-\Sigma_iq_i)\hookrightarrow \calO_Q$ on $Q$. Then $(T,\xi,\gamma)$ is a Prym curve in the sense of \cite[Def.~ 1]{MR2117416} (see also \cite{MR1082361}). Moreover, under the isomorphism between the moduli space $\overline{Pr}_g^+$ of non-trivial genus $g$ Prym curves and the moduli space $\overline{\calR}_g$ of admissible double covers of stable curves of genus $g$ described in \cite[Prop.~ 5]{MR2117416}, the Prym curve $(T,\xi,\gamma)$ corresponds to the admissible discriminant double cover $\widetilde{D}_L  \rightarrow D_L$.
\end{enumerate}
\end{remark}

	We summarize the discussion in the below proposition.
\begin{proposition} \label{proj X 1}
	Let $(X,\tau)$ be a smooth cubic threefold with an involution $\tau$ of non-Eckardt type, and set $L\subset X$ to be the pointwise fixed line. Then the discriminant plane quintic $D_L$ for the projection $\pi_L:\mathrm{Bl}_L\rightarrow \Pi\cong \bP^2$ is the union $D_L=C\cup Q$ of a smooth plane cubic $C$ and a transverse conic $Q$. Moreover, the restriction of the discriminant double cover $\pi_L:\widetilde{D}_L\rightarrow D_L$ to the cubic component $C$ is a double cover $\pi:\widetilde{C}\rightarrow C$ branched in the six intersection points $C\cap Q$ and associated to the line bundle $\calO_C(1)$. \qed
\end{proposition}
	
	\subsection{The intermediate Jacobians of cubic threefolds with a non-Eckardt type involution via the projections from the pointwise fixed line} \label{IJ as prym 1} Let $(X,\tau)$ be a smooth cubic threefold with an involution of non-Eckardt type as in Equations \eqref{eqn:X} and \eqref{eqn:tau}. Consider as in the previous subsection the conic fibration $\pi_L:\mathrm{Bl}_LX\rightarrow \Pi=V(x_3,x_4)\cong \bP^2_{x_0,x_1,x_2}$ obtained by projecting $X$ from the unique pointwise fixed line $L\subset X$. Denote the discriminant double cover by $\pi_L:\widetilde{D}_L\rightarrow D_L$ where $\widetilde{D}_L=\widetilde{C}\cup \widetilde{Q}$ and $D_L=C\cup Q$. Also let $\pi:\widetilde{C}\rightarrow C$ be the restriction of $\pi_L$ to $C$. Since $L$ is fixed, there exists an involution on $\mathrm{Bl}_LX$ induced by $\tau$, which further induces an involution on $\widetilde{D}_L$ commuting with $\pi_L:\widetilde{D}_L\rightarrow D_L$. By abuse of notation, we still denote by $\tau$ the involution on $\widetilde{D}_L$ (and also the involution on the Prym variety $P(\widetilde{D}_L,D_L)$) induced by the non-Eckardt type involution on $X$. Define the invariant part $P(\widetilde{D}_L,D_L)^\tau:=\im(1+\tau)$ and the anti-invariant part $P(\widetilde{D}_L,D_L)^{-\tau}:=\im(1-\tau)$. As recalled in \S\ref{general conic fib}, the intermediate Jacobian $JX$ is canonically isomorphic to the Prym variety $P(\widetilde{D}_L,D_L)$. We now give an explicit description of $P(\widetilde{D}_L,D_L)$ following \cite[\S0.3]{MR472843} which allows us to study the induced involution $\tau$ on $JX\cong P(\widetilde{D}_L,D_L)$.	
	
	Let $\widetilde{\nu}:\widetilde{C}\coprod \widetilde{Q}\rightarrow \widetilde{C}\cup \widetilde{Q}$ and $\nu:C\coprod Q\rightarrow C\cup Q$ be the normalizations of $\widetilde{D}_L=\widetilde{C}\cup \widetilde{Q}$ and $D_L=C\cup Q$ respectively. Let 
	\[\pi'_L:\widetilde{C}\coprod \widetilde{Q}\rightarrow C\coprod Q\] 
be the double cover induced by the discriminant double cover $\pi_L:\widetilde{C}\cup \widetilde{Q}\rightarrow C\cup Q$. Denote the ramification points of $\widetilde{C}\rightarrow C$ (respectively, $\widetilde{Q}\rightarrow Q$) by $\tilde{c}_1,\dots \tilde{c}_6\in \widetilde{C}$ (respectively, $\tilde{q}_1,\dots \tilde{q}_6\in \widetilde{Q}$). Note that $\widetilde{\nu}(\tilde{c}_i)=\widetilde{\nu}(\tilde{q}_i)$ for $i=1,\dots 6$. Following \cite[\S0.3]{MR472843} (see also \S\ref{general conic fib}), we set $H'$ to be the subgroup of $\Pic(\widetilde{C}\coprod\widetilde{Q})$ generated by $\calO(\tilde{c}_i-\tilde{q}_i)$ for $1\leq i\leq 6$. Let $H$ denote the image of $H_0:=H'\cap J(\widetilde{C}\coprod \widetilde{Q})$ in the quotient abelian variety
	\[J(\widetilde{C}\coprod \widetilde{Q})/\pi_L'^*J(C\coprod Q)\cong \left(J(\widetilde{C})/\pi^*J(C)\right)\times J(\widetilde{Q}).\] 
By \cite[Exer.~ 0.3.5]{MR472843}, $H$ consists of 2-torsion elements and is isomorphic to $(\bZ/2\bZ)^4$. Furthermore, $H$ is the kernel of the isogeny of polarized abelian varieties 
	\[\phi:=(\tilde{\nu}^*)^\vee:J(\widetilde{C}\coprod \widetilde{Q})/\pi_L'^*J(C\coprod Q)\cong (J(\widetilde{C})/\pi^*J(C))\times J(\widetilde{Q})\rightarrow P(\widetilde{D}_L,D_L).\]
More precisely, $P(\widetilde{D}_L,D_L)$ admits the principal polarization defined in \cite[Thm.~ 3.7]{MR572974}. Following \cite[Prop.~12.1.3]{MR2062673}, $J(\widetilde{C}\coprod \widetilde{Q})/\pi_L'^*J(C\coprod Q)$ is dual to $P(\widetilde{C}\coprod \widetilde{Q}, C\coprod Q) \cong P(\widetilde{C},C)\times P(\widetilde{Q},Q)$, and the dual polarization on it corresponds to the product of the following polarizations on $J(\widetilde{C})/\pi^*J(C)$ and on $J(\widetilde{Q})$: $J(\widetilde{C})/\pi^*J(C)$ is the dual abelian variety\footnote{See for instance \cite[Prop.~12.1.3]{MR2062673}. Let us also recall the following. Let $\pi:\widetilde{C}\rightarrow C$ be a connected double cover of smooth curves branched in $2r$ points. Then the principal polarization on $J(\widetilde{C})$ induces a polarization on the Prym variety $P(\widetilde{C},C)$ which is of type $(1,\dots,1,2,\dots,2)$ with $1$'s repeated $\max\{0,r-1\}$ times. As the dual abelian variety, $J(\widetilde{C})/\pi^*J(C)$ is equipped with a dual polarization of type $(1,\dots,1,2,\dots,2)$ with $2$'s repeated $\max\{0,r-1\}$ times (compare \cite[\S2]{MR1976843}).} to $P(\widetilde{C},C)$ and therefore comes with a dual polarization, and $J(\widetilde{Q})$ is equipped with twice of the canonical principal polarization. 

\begin{proposition} \label{prymdecomp1}
	Notation as above. There exists an isogeny of polarized abelian varieties
	\[\phi:\left(J(\widetilde{C})/\pi^*J(C)\right)\times J(\widetilde{Q})\rightarrow P(\widetilde{D}_L,D_L)\]
with kernel $H\cong (\bZ/2\bZ)^4$. Moreover, with respect to the action $\tau$ on $P(\widetilde{D}_L,D_L)$ induced by the non-Eckardt type involution on $X$, the isogeny induces isomorphisms of polarized abelian varieties 
	\[P(\widetilde{D}_L,D_L)^\tau\cong J(\widetilde{C})/\pi^*J(C);\,\,\, P(\widetilde{D}_L,D_L)^{-\tau}\cong J(\widetilde{Q}).\]
\end{proposition}
\begin{proof}
	The proof is quite similar to that of \cite[Prop.~3.10]{MR4313238}. The existence of the isogeny $\phi$ and the description of the kernel $H$ is the content of \cite[Prop.~ 0.3.3]{MR472843} (see also \S\ref{general conic fib}). It suffices to prove the assertion regarding the invariant and anti-invariant abelian subvarieties $P(\widetilde{D}_L,D_L)^\tau$ and $P(\widetilde{D}_L,D_L)^{-\tau}$.
	
	Let $\iota=\iota_L$ be the covering involution associated to the double cover $\pi_L:\widetilde{D}_L=\widetilde{C}\cup\widetilde{Q}\rightarrow D_L=C\cup Q$. Consider the involution $\tau:\widetilde{C}\cup\widetilde{Q}\rightarrow \widetilde{C}\cup\widetilde{Q}$ induced by the non-Eckardt type involution on $X$. We claim that the action of $\tau$ on $\widetilde{C}$ is trivial, while the action on $\widetilde{Q}$ coincides with $\iota$. Recall that the curve $\widetilde{D}_L=\widetilde{C}\cup\widetilde{Q}$ parametrizes the residual lines to $L$ in a degenerate fiber of the conic fibration $\pi_L:\mathrm{Bl}_LX\rightarrow \Pi=V(x_3,x_4)\cong \bP^2_{x_0,x_1,x_2}$. Take a point $x\in\Pi$, and consider the plane it corresponds to; namely, the span $\langle L, x\rangle\subset \bP^4$. Since both $L$ and $x\in \Pi$ are fixed by $\tau:[x_0,x_1,x_2,x_3,x_4]\mapsto [x_0,x_1,x_2,-x_3,-x_4]$, the plane $\langle L, x\rangle$ is invariant under $\tau$. Assume now that $x$ is a point on $C$ or $Q$ but not both (otherwise, the claim clearly holds). Then the  $\bP^2$-section $\langle L, x\rangle\cap X$ is three distinct lines, $L\cup M\cup M'$. The lines $M, M'$ correspond to distinct points on $\widetilde{C}\cup\widetilde{Q}$. The involution $\tau$ either interchanges $M$ and $M'$, or it leaves both fixed. In other words, $\tau $ either acts as the identity on a point of $\widetilde{C}\cup\widetilde{Q}$, or by the covering involution $\iota$. The curve $C\subset \bP^2$ parametrizes pairs of lines such that each line is invariant under the involution of non-Eckardt type on $X$ (Lemma \ref{inv lines}), and so the points of $\widetilde{C}$ are fixed by the action of $\tau$. The lines $M,M'$ parametrized by $x\in Q\subset \bP^2$ are not preserved and hence must be interchanged by the involution of non-Eckardt type on $X$; thus $\tau$ acts by the covering involution $\iota$ on $\widetilde{Q}$ as claimed.
	
	It follows that the isogeny $\phi$ is equivariant with respect to the involution $\sigma:=(1,\iota)$ on the product $(J(\widetilde{C})/\pi^*J(C))\times J(\widetilde{Q})$. Since $\im(1+\iota)=(\pi_L|_{\widetilde{Q}})^*J(Q)=0$ on $J(\widetilde{Q})$, the invariant part $\im(1+\sigma)=(J(\widetilde{C})/\pi^*J(C))\times\{0\}$. Similarly, the anti-invariant part $\im(1-\sigma)=\{0\}\times J(\widetilde{Q})$. We then deduce that $\phi$ induces isogenies 
	\[\left(J(\widetilde{C})/\pi^*J(C)\right)\times\{0\}\rightarrow P(\widetilde{D}_L,D_L)^\tau; \,\,\, \{0\}\times J(\widetilde{Q})\rightarrow P(\widetilde{D}_L,D_L)^{-\tau}\]
which are isogenies of polarized abelian varieties since $\phi$ preserves the polarizations. It is not difficult to show that $H\cap ((J(\widetilde{C})/\pi^*J(C))\times\{0\}) = H\cap (\{0\}\times J(\widetilde{Q}))=\{(0,0)\}$ using the description of the kernel $H$ given earlier in this subsection. As a result, the above isogenies are isomorphisms which completes the proof of the proposition.
\end{proof}
	
	Putting the above discussion together, we obtain the following theorem.
\begin{theorem}[Intermediate Jacobian via projecting from the pointwise fixed line] \label{inv/antiinv part1}
	Let $(X,\tau)$ be a smooth cubic threefold with an involution $\tau$ of non-Eckardt type fixing pointwise a line $L\subset X$. Let $JX$ be the intermediate Jacobian which is principally polarized. Projecting $X$ from $L$, we obtain a fibration in conics $\pi_L:\mathrm{Bl}_LX\rightarrow \Pi\cong\bP^2$. Let $\pi_L:\widetilde{D}_L=\widetilde{C}\cup\widetilde{Q}\rightarrow D_L=C\cup Q$ be the discriminant double cover, where $C$ is a smooth cubic and $Q$ is a conic intersecting $C$ transversely. We identify $J(\widetilde{C})/\pi^*J(C)$ as the dual abelian variety to $P(\widetilde{C},C)$ with the dual polarization. We also equip the Jacobian $J(\widetilde{Q})$ with twice of the canonical principal polarization. 
	Then there exists an isogeny of polarized abelian varieties
	\[\phi: \left(J(\widetilde{C})/\pi^*J(C)\right)\times J(\widetilde{Q})\rightarrow JX\] with $\ker\phi\cong (\bZ/2\bZ)^4$.
Moreover, with respect to the action of $\tau$ on $JX$, the isogeny $\phi$ induces isomorphisms of polarized abelian varieties
	\[JX^\tau\cong J(\widetilde{C})/\pi^*J(C); \,\,\, JX^{-\tau}\cong J(\widetilde{Q}).\]
\end{theorem}
\begin{proof}
Since the intermediate Jacobian $JX$ is canonically isomorphic to the Prym variety $P(\widetilde{D}_L,D_L)$, and the involutions are both induced from the non-Eckardt type involution on $X$, the theorem follows from Proposition \ref{prymdecomp1}.
\end{proof}

\section{Global Torelli for cubic threefolds with an involution of non-Eckardt type} \label{torelli}

	Let $\calM$ be the moduli space of cubic threefolds with a non-Eckardt type involution, and set $\calA_3^{(1,2,2)}$ to be the moduli space of abelian threefolds with a polarization of type $(1,2,2)$. Recall that we have defined in \S\ref{period} the following period map:
	\[\calP:\calM \longrightarrow \calA_3^{(1,2,2)} \]
	\[(X,\tau) \mapsto JX^\tau\]
which sends a smooth cubic threefold $X$ with a non-Eckardt type involution $\tau$ to the invariant part $JX^\tau$ of the intermediate Jacobian. The goal of this section is to prove the following global Torelli theorem for $\mathcal{P}$. 

\begin{theorem}[Global Torelli for cubic threefolds with a non-Eckardt type involution] \label{global torelli}
	The period map $\calP:\calM\rightarrow \calA_3^{(1,2,2)}$ for smooth cubic threefolds with a non-Eckardt type involution is injective.
\end{theorem}

	Let us briefly outline the strategy. Let $(X,\tau)$ be a smooth cubic threefold with an involution of non-Eckardt type. We have shown in \S\ref{projX1} that projecting $X$ from the unique pointwise fixed line $L\subset X$ (as in Lemma \ref{inv lines}) determines a double cover $\pi:\widetilde{C}\rightarrow C$ which is the restriction of the discriminant double cover $\pi_L:\widetilde{D}_L=\widetilde{C}\cup \widetilde{Q}\rightarrow D_L=C\cup Q$ to the smooth cubic component $C$. Equivalently, the fibration in conics $\pi_L:\mathrm{Bl}_LX\rightarrow \bP^2$ gives a quadruple $(C,\beta,\calL,s)$ consisting of the branch divisor $\beta=C\cap Q$, the line bundle $\calL=\calO_C(1)$ (see Proposition \ref{thetachar}) associated to the double cover and a section $s=\ell_1(x_0,x_1,x_2)\ell_2(x_0,x_1,x_2)-\ell_3^2(x_0,x_1,x_2)$ (see Equation \eqref{rewrite eqn}) of $\calO_C(\beta)$ vanishing on $\beta=C\cap Q$. Conversely, we will prove in \S\ref{reconstructX} that $(X,\tau)$ can be reconstructed (up to projective equivalence) from the double cover $\pi:\widetilde{C}\rightarrow C$, or equivalently, from the quadruple $(C,\beta,\calL,s)$. Note that here we can also view $C$ as a smooth genus $1$ curve whose embedding into $\bP^2$ as a smooth cubic curve is given by the linear system $|\calL|=|\calO_C(1)|$ (such that the branch divisor $\beta$ lies on a quadratic curve $Q$). In particular, our reconstruction of the cubic threefold $(X,\tau)$ with a non-Eckardt type involution can be thought of as a degenerate case of \cite[\S1.6]{MR472843} and \cite[Prop.~4.2]{MR2123232}. The other key ingredient, which we will recall in \S\ref{injprym}, is the injectivity of the Prym map $\calP_{1,6}:\calR_{1,6}\rightarrow \calA_3^{(1,1,2)}$ for double covers of smooth genus $1$ curves ramified in six distinct points due to Ikeda \cite{MR4156417} and Naranjo and Ortega \cite{NO_torelli}. The proof of Theorem \ref{global torelli} will be completed in \S\ref{pftorelli} using Theorem \ref{inv/antiinv part1} and the above results. The infinitesimal Torelli theorem will be discussed in \S\ref{inftorellinonEc}. 
	
	\subsection{Reconstructing cubic threefolds with a non-Eckardt type involution} \label{reconstructX} We keep notation as in \S\ref{projX1}. From \cite[\S1.6]{MR472843} and \cite[Prop.~4.2]{MR2123232} (see also \S\ref{general conic fib}), we know that the data of a smooth cubic threefold together with a line determines the data of a stable plane quintic curve and an odd theta characteristic (up to projective equivalence), and vice versa. We have a similar result for cubic threefolds with a non-Eckardt type involution.
	
\begin{theorem}[Reconstructing a cubic threefold with a non-Eckardt type involution from a ramified double cover of a genus $1$ curve] \label{reconstruction}
	Given a double cover $\widetilde{C}\rightarrow C$ of a smooth genus $1$ curve $C$ branched in six distinct points, one can associate a smooth cubic threefold $(X,\tau)$ with an involution of non-Eckardt type, such that $\widetilde{C}\rightarrow C$ is obtained by  projecting $X$ from the unique pointwise fixed line $L\subset X$ (i.e.~ we restrict the discriminant double cover of the conic fibration $\pi_L:\mathrm{Bl}_LX\rightarrow \Pi\cong \bP^2$ to the smooth cubic component).
\end{theorem}
\begin{proof}
	The first observation is that a double cover $\widetilde{C}\rightarrow C$ of a smooth genus $1$ curve $C$ branched in six distinct points determines an admissible double cover $\widetilde{D}\rightarrow D$ of a nodal plane quintic $D$ containing $C$ as a component. More precisely, the data of a double cover $\widetilde{C}\rightarrow C$ of a smooth genus $1$ curve $C$ ramified in six distinct points can be described equivalently as a quadruple $(C,\beta,\calL,s)$ where $\beta$ is the branch divisor on $C$ consisting of six distinct points, $\calL$ is a line bundle on $C$ with $\calL^{\otimes 2}\cong \calO_C(\beta)$ and $s$ is a section of $\calO_C(\beta)$ vanishing on $\beta$. We claim that the data of such a quadruple $(C,\beta,\calL,s)$ are equivalent to the data of a pair $(C\subset \bP^2, Q\subset \bP^2)$ recording the embedding of a smooth cubic $C$ and a transverse conic $Q$ in $\bP^2$. Indeed, the line bundle $\calL$ is very ample (since $\deg\calL \geq 2g(C)+1$), and hence defines an embedding $C\hookrightarrow \bP^2$ of $C$ as a plane cubic,  such that $\calO_{\bP^2}(1)|_C\cong \mathcal L$. From the identification $H^0(\mathbb P^2,\mathcal O_{\mathbb P^2}(2))\stackrel{\sim}{\rightarrow}H^0(C,\mathcal O_{\mathbb P^2}(2)|_C)= H^0(C,\mathcal O_C(\beta))$, we see that there is a unique conic $Q$ in the plane so that $Q\cap C=\beta$.  Since $\beta$ consists of six distinct points, $Q$ is reduced. In summary, we let $D:=C\cup Q$ which is a nodal plane quintic, and set $\widetilde{D}\rightarrow D$ to be the unique double cover branched at the nodes of $D$, such that the restriction to $C$ is the cover $\widetilde C\rightarrow C$ (note that, since $Q$ is rational, only one irreducible double cover of it exists, branched at $C\cap Q$).

	If we knew that the cover $\widetilde{D}\rightarrow D$ were odd, then \cite[\S1.6]{MR472843} (see also \cite[Prop.~4.2]{MR2123232}) would provide a smooth cubic threefold $X$ and a line $L\subset X$ so that projection from $L$ gave the discriminant cover $\widetilde{D}\rightarrow D$.  However, we would still need to show that $X$ was a cubic threefold with non-Eckardt type involution, and that $L$ was the unique pointwise fixed line.  

	For this reason, we instead construct directly the smooth cubic threefold $X$ with a non-Eckardt type involution $\tau$ so that projection from the pointwise fixed line gives the desired cover. To begin, since the rank of $Q$ is $2$ or $3$, it is given by  
	\[\det\left(\begin{matrix}
	\ell_1(x_0,x_1, x_2) & \ell_3(x_0,x_1,x_2)\\
     	\ell_3(x_0,x_1,x_2) &\ell_2(x_0, x_1, x_2)
	\end{matrix}\right)=0\]
for some choice of coordinates on $\mathbb P^2$ and some linear forms $\ell_i(x_0,x_1,x_2)$. Let $g(x_0,x_1,x_2)=0$ be the equation for $C$. Take $X$ to be the cubic threefold in $\bP^4$ given by 
\begin{equation} \label{X}
	\ell_1(x_0,x_1,x_2)x_3^2+\ell_2(x_0,x_1,x_2)x_4^2+2\ell_3(x_0,x_1,x_2)x_3x_4+g(x_0,x_1,x_2)=0
\end{equation} 
which admits an involution $\tau:[x_0,x_1,x_2,x_3,x_4]\mapsto [x_0,x_1,x_2,-x_3, -x_4]$ of non-Eckardt type. Let $L:=V(x_0,x_1,x_2)$ and $\Pi:=V(x_3,x_4)\cong \bP^2$. Taking the determinant of the matrix in Equation \eqref{E:Mdef}, one has that $D= C\cup Q$ is the discriminant curve of the conic fibration obtained by projecting $X$ from $L$ (compare \S\ref{projX1}). 

	It remains to show that the cubic threefold $X$ is smooth, and that the associated discriminant double cover obtained by projection from $L$ is the cover $\widetilde{D}\rightarrow D$ constructed above. Assuming $X$ is smooth, the latter assertion is Proposition \ref{thetachar}. The smoothness of $X$ will follow from the fact that the cubic $C\subset \bP^2$ is smooth and the conic $Q$ intersects $C$ transversely. Similar arguments are made in the proof of \cite[Prop.~4.2]{MR2123232}. Specifically, suppose that $a=[a_0,a_1,a_2,a_3,a_4]$ is a singular point of $X$, and consider first the case $a\not\in L=V(x_0,x_1,x_2)$. After a change of coordinates fixing $x_0,x_1,x_2$, we may assume that $a_3=a_4=0$. Analyzing the partial derivatives of Equation \eqref{X}, we see that the cubic $C$ is singular at the point $[a_0,a_1,a_2]$, which is a contradiction. Next, let us consider the case when the point $a$ lies in the fixed line $L$. We may assume that $a=[0,0,0,1,0]$. A direct calculation using Equation \eqref{X} shows that the conic $Q$ is non-reduced, which is absurd. 	
\end{proof}

	\subsection{Injectivity of the Prym map for double covers of genus $1$ curves branched in six points} \label{injprym} Consider the Prym map $\calP_{g,2n}:\calR_{g,2n}\rightarrow \calA_{g-1+n}^{(1,\dots,1,2,\dots,2)}$ (with $2$'s repeated $g$ times); recall that $\calR_{g,2n}$ is the moduli space of double covers of smooth genus $g$ curves branched in $2n$ distinct points, and $\calA_{g-1+n}^{(1,\dots,1,2,\dots,2)}$ is the moduli space of abelian varieties of dimension $g-1+n$ with a polarization of type $(1,\dots,1,2,\dots,2)$. In \cite{MR4156417}, Ikeda studies double covers of elliptic curves and proves the injectivity of the Prym map $\calP_{1,2n}$ with $n\geq 3$. More generally, Naranjo and Ortega prove the Prym-Torelli theorem for $\calP_{g,2n}$ with $g>0$ and $n\geq 3$ in \cite{NO_torelli}.
	 
\begin{theorem}[\cite{MR4156417} Theorem 1.2; \cite{NO_torelli} Theorem 1.1] \label{injective for double covers}
	If $n\geq 3$, then the Prym map $\calP_{1,2n}:\calR_{1,2n}\rightarrow \calA_{n}^{(1,\dots,1,2)}$ is injective. In particular, $\calP_{1,6}:\calR_{1,6}\rightarrow \calA_3^{(1,1,2)}$ is injective.
\end{theorem}	

	Let us very briefly review the proof of the above theorem following \cite{MR4156417} and \cite{NO_torelli}. Let $(P(\widetilde{C},C), \Xi)$ denote the Prym variety for a double cover $\widetilde{C}\rightarrow C$ of a smooth genus $1$ curve $C$ branched in $2n$ distinct points. For any member $\Sigma\in |\Xi|$, Ikeda studies the Gauss map $\Psi_\Sigma$ (more precisely, the branch locus of the restriction $\Psi_\Sigma|_{\mathrm{BS}|\Xi|}$ of $\Psi_\Sigma$ to the base locus $\mathrm{BS}|\Xi|$ of $|\Xi|$), and uses the information to specify a divisor $\Sigma_0\in |\Xi|$ which allows him to reconstruct the double cover $\widetilde{C}\rightarrow C$. Ikeda's proof is a generalization of the proof of the Torelli theorem for hyperelliptic curves due to Andreotti (cf.~\cite{MR102518}). The approach of Naranjo and Ortega is different and relies on the description of the base locus of $|\Xi|$ given in \cite{MR3896124} and on a generalized Torelli theorem proved by Martens in \cite{MR182624}. To be specific, Naranjo and Ortega recover a certain Brill-Noether locus on $\Pic^0(\widetilde{C})$ together with an involution through a birational model of the base locus $\mathrm{BS}|\Xi|$; by Martens' generalized Torelli, the Brill-Noether locus determines the double cover $\widetilde{C}\rightarrow C$. More generally, Naranjo and Ortega's argument can be used to prove the injectivity of the Prym map $\calP_{g,2n}$ with $g>0$ and $n\geq 3$.  

	\subsection{Proving global Torelli for cubic threefolds with a non-Eckardt type involution} \label{pftorelli}
\begin{proof}[Proof of Theorem \ref{global torelli}]
	Let $(X_1, \tau_1)$ and $(X_2,\tau_2)$ be two cubic threefolds with involutions of non-Eckardt type.  We will prove that if the invariant parts of the intermediate Jacobian are isomorphic to each other $JX_1^{\tau_1}\cong JX_2^{\tau_2}$, then $(X_1,\tau_1)$ is projectively equivalent to $(X_2,\tau_2)$. Recall from Proposition \ref{proj X 1} that each involution $\tau_i$ fixes pointwise a line $L_i\subset C_i$, such that when projecting $X_i$ from $L_i$ we get a double cover $\pi_i:\widetilde{C}_i\rightarrow C_i$ in $\calR_{1,6}$. To be more precise, the double cover $\widetilde{C}_i\rightarrow C_i$ is obtained by projecting $X_i$ from $L_i$ and restricting the discriminant pseudo-double cover $\pi_i:\widetilde{C}_i\cup\widetilde{Q}_i\rightarrow C_i\cup Q_i$ over the smooth cubic component $C_i$. By Theorem \ref{inv/antiinv part1}, we have $JX_i^{\tau_i}\cong (P(\widetilde{C}_i, C_i))^\vee$ as polarized abelian varieties which implies that $P(\widetilde{C}_1, C_1)\cong P(\widetilde{C}_2, C_2)$ (see also \cite[Thm.~3.1]{MR1976843}). From Theorem \ref{injective for double covers}, we know that the two double covers $\pi_i:\widetilde{C}_i\rightarrow C_i$ are equivalent. By Theorem \ref{reconstruction}, we can reconstruct $(X_i,\tau_i)$ from the data of the double cover $\widetilde{C}_i\rightarrow C_i$ branched in six points. It follows that the cubic threefolds $(X_1,\tau_1)$ and $(X_2,\tau_2)$ are projectively equivalent.	 
\end{proof}

	\subsection{Infinitesimal Torelli for cubic threefolds with an involution of non-Eckardt type} \label{inftorellinonEc} Recall that $\calM$ is the moduli space of cubic threefolds $(X,\tau)$ with an involution of non-Eckardt type. Let us consider the open subset $\calM_0\subset \calM$ parametrizing those $(X,\tau)$ satisfying that when projecting $X$ from the unique fixed line $L\subset X$ (as in \S\ref{projX1}), the conic component $Q$ of the discriminant curve $D_L$ is smooth. More concretely, the equation of a member in the complement $\calM\backslash\calM_0$ can be written as Equation \eqref{rewrite eqn} with $\ell_3(x_0,x_1,x_2)\equiv 0$ (i.e.~no terms containing $x_3x_4$). 
		
\begin{proposition}[Infinitesimal Torelli for cubic threefolds with a non-Eckardt type involution] \label{inftorelli}
	Let $\calP:\calM\rightarrow \calA_3^{(1,2,2)}$ be the period map for cubic threefolds with a non-Eckardt type involution. The differential $d\calP$ is an isomorphism at every point of $\calM_0\subset \calM$. (As a consequence, $\calP|_{\calM_0}:\calM_0\rightarrow \calA_3^{(1,2,2)}$ is an open embedding.)
\end{proposition}
\begin{proof}
	The infinitesimal computation is quite similar to that of \cite[\S5.1]{MR4313238} (see also \cite[\S8.1, \S8.3]{MR3727160}). Specifically, let $(X,\tau)$ be a smooth cubic threefold with an involution $\tau$ of non-Eckardt type cut out by $F=0$ as in Equation \eqref{eqn:X}. On one side, we have: 
	\[T_{(X,\tau)}\calM=(R_F^3)^\tau\] where $R_F^i$ ($i\geq 0$) denotes degree $i$ part of the Jacobian ring of $F$ and the superscript means taking the $\tau$-invariant subspace. On the other side, we also have: 
	\[T_{JX^\tau}\calA_3^{(1,2,2)}=\Sym\Hom(H^{2,1}(X)^\tau, H^{1,2}(X)^\tau)=\Sym\Hom((R_F^1)^\tau, (R_F^4)^\tau).\]  
Furthermore, the differential $d\calP_{(X,\tau)}$ of the period map $\calP: \calM \rightarrow \calA_3^{(1,2,2)}$ can be identified with the following map, which is induced by cup product: 
	\[(R_F^3)^\tau \rightarrow \Sym\Hom((R_F^1)^\tau, (R_F^4)^\tau).\] 	 
Using Macaulay's theorem to identify $R_F^4$ with $(R_F^1)^\vee$, the differential is 
	\[d\calP_{(X,\tau)}:(R_F^3)^\tau \rightarrow \Sym^2((R_F^1)^\tau)^\vee\] 
and the codifferential is 
	\[d\calP^*_{(X,\tau)}:H^0(\bP^2_{x_0,x_1,x_2}, \calO(2))\cong \Sym^2((R_F^1)^\tau) \rightarrow (R_F^2)^\tau.\] The kernel of the codifferential $d\calP^*_{(X,\tau)}$ is $J_F^2\cap \Sym^2\bC[x_0,x_1,x_2]$, where $J_F^2$ denotes degree $2$ part of the Jacobian ideal. Now we suppose that $(X,\tau)\in \calM_0$. Rewrite the equation $F$ of $X$ in the form of Equation \eqref{rewrite eqn}
	\[\ell_1(x_0,x_1,x_2)x_3^2+\ell_2(x_0,x_1,x_2)x_4^2+2\ell_3(x_0,x_1,x_2)x_3x_4+g(x_0,x_1,x_2)=0.\] 
A direct calculation then shows that $J_F^2\cap \Sym^2\bC[x_0,x_1,x_2]$ is not trivial if and only if $\ell_1(x_0,x_1,x_2)$, $\ell_2(x_0,x_1,x_2)$ and $\ell_3(x_0,x_1,x_2)$ have a common zero. However, it is impossible since otherwise the conic component $Q$ of the discriminant curve $D_L$ for the fibration in conics $\pi_L:\mathrm{Bl}_LX\rightarrow \Pi$ will be singular. Note also that $\dim\Sym^2((R_F^1)^\tau)=\dim(R_F^2)^\tau=6$. It follows that the (co)differential of the period map $\calP:\calM\rightarrow \calA_3^{(1,2,2)}$ at every point of $\calM_0\subset \calM$ is an isomorphism.   
\end{proof}

\begin{remark} \label{M0hyperelliptic}
	We remark that the (co)differential of the period map $\calP:\calM\rightarrow \calA_3^{(1,2,2)}$ has $1$-dimensional kernel at a point $(X,\tau)$ of the divisor $\calM\backslash \calM_0$ (note that a quadratic curve being degenerate is a codimension $1$ condition) and hence the infinitesimal Torelli fails. A similar phenomenon happens for the hyperelliptic locus when considering the period map for smooth curves of genus greater than $2$ (see for instance \cite[\S1]{MR756850}). We give the following explanation using \cite[\S1]{MR637511}. Observe that a member $(X,\tau)\in \calM\backslash \calM_0$ is cut out by 
	\[\ell_1(x_0,x_1,x_2)x_3^2+\ell_2(x_0,x_1,x_2)x_4^2+g(x_0,x_1,x_2)=0\]
and hence admits extra automorphisms $\sigma:x_3\mapsto -x_3$ and $\sigma':x_4\mapsto -x_4$ (note that $\tau=\sigma\circ\sigma'=\sigma'\circ\sigma$). By \cite[\S1]{MR637511} (see also \cite[\S1]{MR756850}), the differential $d\calP_{(X,\tau)}: T_{(X,\tau)}\calM\rightarrow T_{JX^\tau}\calA_3^{(1,2,2)}$ is equivariant with respect to the action of $\sigma$ (and also $\sigma'$). Using the identifications in the proof of Proposition \ref{inftorelli}, it is not difficult to see that 
$T_{JX^\tau}\calA_3^{(1,2,2)}$ is fixed by $\sigma$ while $T_{(X,\tau)}\calM$ splits as the direct sum of the $(+1)$ eigenspace $T_{(X,\tau)}^+\calM$ (which is $5$-dimensional and corresponds to infinitesimal deformations inside $\calM\backslash \calM_0$) and the $(-1)$ eigenspace $T_{(X,\tau)}^-\calM$ (which has dimension $1$ and represents infinitesimal deformations in the normal directions). As a result, the $(-1)$ eigenspace $T_{(X,\tau)}^-\calM$ must be contained in the kernel of $d\calP_{(X,\tau)}$ and therefore the infinitesimal Torelli fails. 
\end{remark}

\section{Cubic threefolds with an involution of non-Eckardt type as fibrations  in conics II: invariant lines} \label{conic fib 2}

	In this section, we aim to characterize the intermediate Jacobians $JX$ of general cubic threefolds $X$ with an involution $\tau$ of non-Eckardt type via projections from general invariant lines $l\subset X$ that are not pointwise fixed. Specifically, we project a cubic threefold $(X,\tau)$ with a non-Eckardt type involution from a $\tau$-invariant line $l$ and study the obtained conic fibration in \S\ref{projX2}. Among others, we prove that the involution of non-Eckardt type on $X$ induces an involution on the discriminant quintic curve which is generically smooth. A different description of the invariant part $JX^\tau$ and the anti-invariant part $JX^{-\tau}$ for a general $(X,\tau,l)$ is then given in \S\ref{IJ as prym 2} using the techniques developed in \cite{MR0379510}, \cite{MR1188194} and \cite{MR2218008}. We conclude by giving an application of our results in the study of the generic fiber of the Prym map $\calP_{2,4}:\calR_{2,4}\rightarrow \calA_3^{(1,2,2)}$ in \S\ref{genericfiber}. In particular, the discussion in \cite[\S4]{NO_torelli} and \cite[\S5]{MR4435960} will play a crucial role. 

	\subsection{Projecting cubic threefolds with a non-Eckardt type involution from invariant lines} \label{projX2} We keep notation as in \S\ref{geometry}. Let $(X,\tau)$ be a smooth cubic threefold with a non-Eckardt type involution as in Equations \eqref{eqn:X} and \eqref{eqn:tau}
	\[x_0q_0(x_3, x_4)+x_1q_1(x_3, x_4)+x_2q_2(x_3,x_4)+g(x_0,x_1,x_2)=0;\]
	\[\tau:[x_0,x_1,x_2,x_3,x_4]\mapsto[x_0,x_1,x_2,-x_3,-x_4].\] 
Let $L=V(x_0,x_1,x_2)\subset X$ denote the pointwise fixed line under $\tau$. When studying the projection of $X$ from $L$, it is more convenient to rewrite the above equation in the form of Equation \eqref{rewrite eqn} 
	\[\ell_1(x_0,x_1,x_2)x_3^2+\ell_2(x_0,x_1,x_2)x_4^2+2\ell_3(x_0,x_1,x_2)x_3x_4+g(x_0,x_1,x_2)=0.\]
Let $C=V(g(x_0,x_1,x_2),x_3,x_4)\subset \Pi=V(x_3,x_4)\cong \bP^2_{x_0,x_1,x_2}$ be the fixed plane section. By  Lemma \ref{inv lines}, an invariant line $l\subset X$ that is different from $L$ intersects both $L$ and $C\subset \Pi\cong \bP^2_{x_0,x_1,x_2}$. From the discussion in \S\ref{projX1} (see also the proof of Proposition \ref{prymdecomp1}), we know that $l$ is contained in a singular fiber of the projection $\pi_L:\mathrm{Bl}_LX\rightarrow \Pi$ over a point of the cubic component $C$ of the discriminant curve. In other words, denote the restriction of the discriminant double cover to $C$ by $\pi:\widetilde{C}\rightarrow C$; invariant lines in $X$ that are not pointwise fixed are then parametrized by $\widetilde{C}$ (which is a smooth curve of genus $4$).
	
	Let $l\neq L$ be an invariant line in $X$. Without loss of generality, we assume that $l$ is 	contained in the fiber of $\pi_L:\mathrm{Bl}_LX\rightarrow \Pi\cong \bP^2_{x_0,x_1,x_2}$ over $[0,0,1]\in C\subset \Pi$. Note that the equation $g(x_0,x_1,x_2)$ of $C\subset \Pi$ has no term $x_2^3$. Also, the quadratic polynomial $q_2(x_3,x_4)$ in Equation \eqref{eqn:X} can be factored as $q_2(x_3,x_4)=(\alpha x_3+\beta x_4)(\gamma x_3+\delta x_4)$ with $\alpha$, $\beta$, $\gamma$ and $\delta$ constants. Consider the fiber of $\pi_L:\mathrm{Bl}_LX\rightarrow \Pi$ over the point $[0,0,1]\in C\subset \Pi$. The equations for the invariant lines $l$ and $l'$ contained in the fiber are respectively
	\[l=V(x_0,x_1,\alpha x_3+\beta x_4);\,\,\, l'=V(x_0,x_1,\gamma x_3+\delta x_4).\]	
Let us set $z:=\alpha x_3+\beta x_4$ and apply the change of coordinates 
	\[[x_0,x_1,x_2,x_3,x_4]\mapsto [x_0,x_1,x_2,z,x_4].\]
Then the line $l=V(x_0,x_1,z)$ and the involution $\tau$ acts by 
\[\tau:[x_0,x_1,x_2,z,x_4]\mapsto[x_0,x_1,x_2,-z,-x_4].\]

	Now we project $X$ from the invariant line $l=V(x_0,x_1,z)$ to the complementary plane $\bP^2_l:=V(x_2,x_4)\cong \bP^2_{x_0,x_1,z}$ and study the corresponding fibration in conics $\pi_l:\mathrm{Bl}_lX\rightarrow \bP^2_l$. Specifically, let us rewrite the equation of $X$ as
\begin{equation} \label{eqn2}
\begin{aligned}
	&L_1(x_0,x_1,z)x_2^2+L_2(x_0,x_1,z)x_4^2+2L_3(x_0,x_1,z)x_2x_4\\
	&+2Q_1(x_0,x_1,z)x_2+2Q_2(x_0,x_1,z)x_4+G(x_0,x_1,z)=0,
\end{aligned}
\end{equation}
where $L_i(x_0,x_1,z)$ are linear polynomials, $Q_j(x_0,x_1,z)$ are quadratic forms and $G(x_0,x_1,z)$ is homogeneous of degree $3$. We can simplify Equation \eqref{eqn2} in the following way.
\begin{itemize}
	\item Since the equation of $X$ is preserved under $\tau,$ both of the terms $L_1(x_0,x_1,z)$ and $L_2(x_0,x_1,z)$ must have no $z$ term; in other words, $L_1(x_0,x_1,z)=L_1(x_0,x_1)$ and $L_2(x_0,x_1,z)=L_2(x_0,x_1)$. For the same reason, $L_3(x_0,x_1,z)$ must be linear in $z$, i.e. $L_3(x_0,x_1,z)=Az$ for some constant $A$. Furthermore, when $l=l'$ (equivalently, $\pi:\widetilde{C}\rightarrow C$ is ramified at the point corresponding to the invariant line $l$), one deduces from the above calculation that $A=0$. 
	\item The terms $Q_1(x_0,x_1,z)x_2$ and $Q_2(x_0,x_1,z)x_4$ must be invariant under $\tau$, so $Q_1(x_0,x_1,z)$ can only contain terms with even powers of $z$ and $Q_2(x_0,x_1,z)=zN(x_0,x_1)$ for a linear form $N(x_0,x_1)$.
	\item Similarly, the cubic polynomial $G(x_0,x_1,z)$ has no monomials with odd powers of $z$.
\end{itemize} 
Thus Equation \eqref{eqn2} becomes
\begin{equation} \label{neweqn}
\begin{aligned}
	&L_1(x_0,x_1)x_2^2+L_2(x_0,x_1)x_4^2+2Azx_2x_4\\&+2Q_1(x_0,x_1,z)x_2+ 2zN(x_0,x_1)x_4+G(x_0,x_1,z)=0;
\end{aligned}	
\end{equation}
in particular, $Q_1(x_0,x_1,z)$ and $G(x_0,x_1,z)$ only contain monomials with even powers of $z$. The matrix associated with the fibration in conics $\pi_l:\mathrm{Bl}_lX\rightarrow \bP^2_l\cong \bP^2_{x_0,x_1,z}$ is 
\begin{equation} \label{eqn:M}
	M=\left(
	\begin{matrix}
	L_1(x_0,x_1) & Az & Q_1(x_0,x_1,z)\\
	Az& L_2(x_0,x_1) & zN(x_0,x_1)\\
	Q_1(x_0,x_1,z) & zN(x_0,x_1) & G(x_0,x_1,z)
	\end{matrix}\right),
\end{equation}
and the discriminant quintic curve $D_l\subset \bP^2_l\cong \bP^2_{x_0,x_1,z}$ is cut out by
	\[
	\begin{aligned}
	&\det(M)=L_1(x_0,x_1)L_2(x_0,x_1)G(x_0,x_1,z)+2Az^2N(x_0,x_1)Q_1(x_0,x_1,z)\\
	&-L_2(x_0,x_1)Q_1^2(x_0,x_1,z)-z^2L_1(x_0,x_1)N^2(x_0,x_1)-A^2z^2G(x_0,x_1,z)=0.
	\end{aligned}
	\]

	The involution $\tau$ of non-Eckardt type on $X$ induces an involution on $\bP^2_l\cong \bP^2_{x_0,x_1,z}$, given by $\tau_{\bP^2_l}:[x_0,x_1,z]\mapsto [x_0,x_1,-z]$. Because $l\subset X$ is an invariant line, there is an induced involution on $\mathrm{Bl}_lX$, which we also denote using $\tau$, making $\pi_l:\mathrm{Bl}_lX\rightarrow \bP^2_l$ equivariant. Let $\pi_l:\widetilde{D}_l\rightarrow D_l$ be the discriminant double cover. Note that the equation of the discriminant curve $D_l$ only contains terms with even powers of $z$ and hence is preserved by $\tau_{\bP^2_l}$. As a result, the restrictions of $\tau$ and $\tau_{\bP^2_l}$ to $\widetilde{D}_l$ and $D_l$ respectively induce involutions
	\[\tau:\widetilde{D}_l\rightarrow \widetilde{D}_l; \,\,\, \tau_{D_l}:D_l\rightarrow D_l\]
with respect to which $\pi_l:\widetilde{D}_l\rightarrow D_l$ is equivariant.
	
	We summarize the discussion above in the following proposition.
\begin{proposition} \label{tauDl}
	Let $(X,\tau)$ be a smooth cubic threefold with an involution $\tau$ of non-Eckardt type. Choose an invariant (but not pointwise fixed) line $l\subset X$ and project $X$ from $l$. Denote the obtained discriminant double cover by $\pi_l:\widetilde{D}_l\rightarrow D_l$. Then the non-Eckardt type involution on $X$ induces involutions $\tau$ and $\tau_{D_l}$ on $\widetilde{D}_l$ and $D_l$ respectively, making $\pi_l:\widetilde{D}_l\rightarrow D_l$ equivariant. 
\end{proposition}	

	\subsection{The intermediate Jacobians of cubic threefolds with a non-Eckardt type involution via the projections from invariant lines} \label{IJ as prym 2} In what follows, we suppose that $(X,\tau)$ is a general cubic threefold admitting a non-Eckardt type involution. It is not difficult to show that there exists an invariant line $l\subset X$ such that the rank of the matrix $M$ in Equation \eqref{eqn:M} never drops to $1$ (e.g.~ consider $V(\det(M_{1,1}), \det(M_{2,3}),\det(M_{3,3}))$ where $M_{i,j}$ denotes the $(i,j)$-minor of $M$), and therefore the discriminant curve $D_l$ is smooth and the discriminant double cover $\pi_l:\widetilde{D}_l\rightarrow D_l$ is connected and \'etale. Recall that $\widetilde{C}$, viewed as a curve on the Fano surface $F(X)$, parametrizes $\tau$-invariant lines $l$ that are not pointwise fixed. It follows that $\widetilde{C}$ is not a component of the curve $R\subset F(X)$ corresponding to lines of second type (in the sense of \cite[Def.~6.6]{MR302652}) or a component of the curve $R'\subset F(X)$ parametrizing lines residual to lines of second type. We will call such an invariant line $l\neq L$ a general invariant line (in particular, $\langle L, l\rangle \cap X\neq L+2l$ and thus the coefficient $A\neq 0$ in Equation \eqref{neweqn}). 
	
	Suppose that $(X,\tau,l)$ is general as above. We wish to study the intermediate Jacobian $JX$ via the Prym variety $P(\widetilde{D}_l,D_l)$. The key observation is that the covering curve $\widetilde{D}_l$ admits two commuting involutions: the involution $\tau$ in Proposition \ref{tauDl} induced from the non-Eckardt type involution on $X$, and the covering involution $\iota=\iota_l$ associated with $\pi_l:\widetilde{D}_l\rightarrow D_l$. In other words, the automorphism group $\Aut(\widetilde{D}_l)$ contains the Klein four group $\langle\tau, \iota\rangle$. This allows us to apply the techniques that go back to \cite{MR0379510}, and explored in more depth in \cite{MR1188194} and \cite{MR2218008}, to decompose $P(\widetilde{D}_l,D_l)$. In particular, we follow closely the analogous case for Eckardt cubic threefolds as discussed in \cite[\S7.1]{MR4313238}.
	
	We will take the quotient of $\widetilde{D}_l$ by an element $g$ of the Klein four group $\langle \tau, \iota\rangle\subset \Aut(\widetilde{D}_l)$; let us denote the quotient curve by 
	\[D_g:=\widetilde{D}_l/\langle g\rangle,\] 
noting that $D_\iota=D_l$.	 Also set $\overline{D}_l:=D_l/\tau_{D_l}$.
\begin{proposition} \label{kleincover}
	We have the following commutative diagram.
\begin{equation*}
\begin{tikzcd}
		&\widetilde{D}_l\arrow{d}{a_{\tau\iota}} \arrow[swap]{dl}{a_\tau} \arrow{dr}{a_\iota=\pi_l}&\\
		D_\tau\arrow[swap]{dr}{b_\tau}& D_{\tau\iota} \arrow{d}{b_{\tau\iota}}& D_\iota=D_l \arrow{dl}{b_\iota}\\
		&\overline{D}_l &
\end{tikzcd}
\end{equation*} 
Moreover,
\begin{enumerate}
	\item The map $a_\tau$ is a double covering map ramified in four points, the map $a_{\tau\iota}$ is a double cover branched in eight points, whereas the map $a_\iota$ is the discriminant double covering map $\pi_l$ which is \'etale.
	\item The map $b_\tau$  is a double covering map ramified in four points, the map $b_{\tau\iota}$ is a double cover branched in two points, and $b_\iota$ is a ramified double cover with six branch points.
	\item The curves are all smooth and their genera are as follows: $g(\widetilde{D}_l)=11$, $g(D_\tau)=5$, $g(D_{\tau\iota})=4$, $g(D_l)=6$ and $g(\overline{D}_l)=2$.
	\end{enumerate}
\end{proposition}
\begin{proof}
	We choose coordinates and project a general cubic threefold $(X,\tau)$ admitting an involution of non-Eckardt type from a general invariant line $l\subset X$ as in \S\ref{projX2}. In particular, the equations of $X$ and the discriminant quintic $D_l\subset \bP^2_{x_0,x_1,z}$ are given in Equations \eqref{neweqn} and \eqref{eqn:M} respectively. Note that $D_l$ is smooth and the discriminant double cover $\pi_l:\widetilde{D}_l\rightarrow D_l$ is connected and \'etale. Also, the involution $\tau_{D_l}:[x_0,x_1,z]\mapsto [x_0,x_1,-z]$ fixes six points on $D_l$: the point $[0,0,1]$ and the five intersection points of $D_l$ and the line $z=0$ which satisfy the equation $z=L_2(x_0,x_1)(L_1(x_0,x_1)G(x_0,x_1,z)-Q_1^2(x_0,x_1,z))=0$ and are distinct since $(X,\tau,l)$ is general. The six pairs of lines in the fibers of $\pi_l:\mathrm{Bl}_lX\rightarrow \bP^2_{x_0,x_1,z}$ that lie over the fixed points on $D_l$ correspond to the ramification points of $a_\tau:\widetilde{D}_l\rightarrow D_\tau$ and $a_{\tau\iota}:\widetilde{D}_l\rightarrow D_{\tau\iota}$, depending on whether they are fixed or switched by $\tau$. By a straightforward computation, one verifies that the four lines corresponding to the preimages of the points $[0,0,1]$ and $V(z,L_2(x_0,x_1))$ under $\pi:\widetilde{D}_l\rightarrow D_l$ are fixed by $\tau$, while the remaining eight lines are interchanged pairwise by $\tau$. For example, the lines lying over $[0,0,1]\in D_l$ are the pointwise fixed line $L$ and another invariant line $l'\neq l$.  The remainder of the assertions follow directly from Riemann-Hurwitz and \cite[Thm.~ 3.2]{MR2218008}.  
\end{proof}

	Applying \cite[Thm.~3.2]{MR2218008} (see also \cite[Prop.~ 7.12]{MR4313238}), we obtain the following description of the invariant part $P(\widetilde{D}_l,D_l)^\tau:=\im(1+\tau)$ and anti-invariant part $P(\widetilde{D}_l,D_l)^{-\tau}:=\im(1-\tau)$ for the induced involution $\tau$ on $P(\widetilde{D}_l,D_l)$.
\begin{proposition}[\cite{MR2218008} Theorem 3.2] \label{prymdecomp2}
	Notation as in Proposition \ref{kleincover}. Consider the Prym variety $P(\widetilde{D}_l,D_l)$ which is principally polarized. There is an isogeny of polarized abelian varieties
	\[\phi_l:P(D_\tau,\overline{D}_l)\times P(D_{\tau\iota},\overline{D}_l)\rightarrow P(\widetilde{D}_l,D_l),\,\,\, (y_1,y_2)\mapsto a_\tau^*(y_1)+a_{\tau\iota}^*(y_2)\] 
with $\ker(\phi_l)\cong (\bZ/2\bZ)^4,$ where $a_\tau^*$ and $a_{\tau\iota}^*$ denote the pull-back maps between the appropriate Jacobians.
	Moreover, with respect to the action of $\tau$ on $P(\widetilde{D}_l,D_l)$, the isogeny $\phi_l$ induces isomorphisms of polarized abelian varieties \[P(\widetilde{D}_l,D_l)^\tau\cong P(D_\tau,\overline{D}_l);\,\,\, P(\widetilde{D}_l,D_l)^{-\tau}\cong P(D_{\tau\iota},\overline{D}_l).\]
\end{proposition}

	Putting it together, we obtain the following theorem.
\begin{theorem}[Intermediate Jacobian via projecting from a general invariant line] \label{inv/antiinv part2}
	Let $(X,\tau)$ be a general cubic threefold with an involution $\tau$ of non-Eckardt type, and choose a general invariant line $l\subset X$. Project $X$ from $l$ and denote the discriminant double cover by $\pi_l:\widetilde{D}_l\rightarrow D_l$, and keep notation as in Proposition \ref{kleincover}. There is an isogeny of polarized abelian varieties
	\[\phi_l: P(D_\tau,\overline{D}_l)\times P(D_{\tau\iota},\overline{D}_l)\rightarrow JX\]
with $\ker\phi\cong (\bZ/2\bZ)^4$.
	Moreover, with respect to the action of $\tau$ on the principally polarized intermediate Jacobian $JX$, the isogeny $\phi_l$ induces isomorphisms of polarized abelian varieties
	\[JX^\tau\cong  P(D_\tau,\overline{D}_l); \,\,\, JX^{-\tau}\cong P(D_{\tau\iota},\overline{D}_l).\]
\end{theorem}
\begin{proof}
The proof is identical to that of Theorem \ref{inv/antiinv part1}, and follows from Proposition \ref{prymdecomp2}.
\end{proof}

	\subsection{The generic fiber of the Prym map for double covers of genus $2$ curves ramified in four points} \label{genericfiber} Consider the Prym map $\calP_{2,4}:\calR_{2,4}\rightarrow \calA_3^{(1,2,2)}$ where $\calR_{2,4}$ is the moduli space of double covers of smooth genus $2$ curves branched in four distinct points, and $\calA_3^{(1,2,2)}$ denotes the moduli space of abelian threefolds with a polarization of type $(1,2,2)$. From \cite[Thm.~1.2]{NO_torelli} or \cite[Thm.~5.2]{MR4435960}, we know that the generic fiber of $\calP_{2,4}$ is birational to an elliptic curve. By Theorem \ref{global torelli} and Proposition \ref{inftorelli}, a general member of $\calA_3^{(1,2,2)}$ can be realized as the invariant part $JX^\tau$ for a general cubic threefold $X$ with an involution $\tau$ of non-Eckardt type. The goal of this subsection is to give a more concrete description of the generic fiber $\calP_{2,4}^{-1}(JX^\tau)$.

	Let us keep notation as in \S\ref{projX1} and \S\ref{projX2}. Similarly as in \cite[\S5]{MR4313238}, we make the following observation. Let us project a general cubic threefold $(X,\tau)$ with a non-Eckardt type involution from a general invariant line $l\neq L$. The quotient $b_\tau:D_\tau\rightarrow \overline{D}_l$ of the discriminant double cover for the projection $\pi_l:\mathrm{Bl}_lX\rightarrow \bP^2_l$ (see Proposition \ref{kleincover}) is contained in the fiber $\calP_{2,4}^{-1}(JX^\tau)$ because of Theorem \ref{inv/antiinv part2}. We now show that for $\tau$-invariant lines $l$ and $l'$ contained in the same fiber of $\pi_L:\mathrm{Bl}_LX\rightarrow \Pi$ (in other words, there exists a plane $P\subset \bP^4$ with $X\cap P=L\cup l\cup l'$) the quotient discriminant double covers $b_\tau:D_\tau\rightarrow \overline{D}_l$ and $b'_\tau:D'_\tau\rightarrow \overline{D}_{l'}$ are isomorphic and hence correspond to the same element in $\calR_{2,4}$. Our main tools are the bigonal and tetragonal constructions; recall that the bigonal construction is also the key ingredient for proving \cite[Thm.~1.2]{NO_torelli} and \cite[Thm.~5.2]{MR4435960}. We refer the reader to \cite[\S2]{MR1188194} (see also \cite[\S1]{MR4435960}) for the description of the bigonal and tetragonal constructions. 

\begin{proposition} \label{bigonalLl}
	Notation as above. Let $l$ and $l'$ be general $\tau$-invariant lines in $X$ that are contained in the same fiber of $\pi_L:\mathrm{Bl}_LX\rightarrow \Pi$. Consider the projections of $X$ from the pointwise fixed line $L\subset X$ (as in \S\ref{projX1}) and from the invariant lines $l, l'\subset X$ (see \S\ref{projX2}). Denote the discriminant double covers by $\pi_L:\widetilde{C}\cup \widetilde{Q}\rightarrow C\cup Q$, $\pi_l:\widetilde{D}_l\rightarrow D_l$ and $\pi_{l'}:\widetilde{D}_{l'}\rightarrow D_{l'}$ respectively. Also let $\pi:\widetilde{C}\rightarrow C$ be the restriction of $\pi_L$ to the smooth cubic component, and set $b_\tau:D_\tau\rightarrow \overline{D}_l$ (respectively, $b'_\tau:D'_\tau\rightarrow \overline{D}_{l'}$) to be the quotient of $\pi_l$ (respectively, $\pi_{l'}$) by the involution induced from $\tau$ (cf.~Proposition \ref{kleincover}).
\begin{enumerate}
\item The union of lines $l\cup l'$ corresponds to a point $c_{l\cup l'}\in C\subset \Pi$ and hence determines a degree $4$ map $p=p_{l\cup l'}:C\cup Q\rightarrow \bP^1$ given by $\calO_{C\cup Q}(1)(-c_{l\cup l'})$ (i.e.~projecting $C\cup Q$ from $c_{l\cup l'}\in C$ to a complementary line in $\Pi$). Similarly, $D_l$ (respectively, $D_{l'}$) admits a map $q=q_{L\cup l'}:D_l\rightarrow \bP^1$ (respectively, $q'=q_{L\cup l}:D_{l'}\rightarrow \bP^1$) of degree $4$. Then 
	\[\widetilde{C}\cup \widetilde{Q}\stackrel{\pi_L}{\rightarrow} C\cup Q\stackrel{p}{\rightarrow}\bP^1;\,\,\,
	  \widetilde{D}_l\stackrel{\pi_l}{\rightarrow} D_l\stackrel{q}{\rightarrow}\bP^1;\,\,\,
	  \widetilde{D}_{l'}\stackrel{\pi_{l'}}{\rightarrow} D_{l'}\stackrel{q'}{\rightarrow}\bP^1\,\,\,
	\]
are tetragonally related (in other words, the tetragonal construction of one produces the other two).
\item Consider the degree $2$ map obtained as the restriction of $p=p_{l\cup l'}:C\cup Q\rightarrow \bP^1$ to the smooth cubic curve $C$ and still use $p$ to denote it. Note that $\overline{D}_l$ and $\overline{D}_{l'}$ are smooth of genus $2$ and hence admit degree $2$ maps to $\bP^1$ defined by the canonical linear systems (which are the unique $g_2^1$'s on $\overline{D}_l$ and $\overline{D}_{l'}$). Then the bigonal construction takes both 
	\[D_\tau\stackrel{b_\tau}{\rightarrow} \overline{D}_l\stackrel{r}{\rightarrow} \bP^1;\,\,\,D'_\tau\stackrel{b'_\tau}{\rightarrow} \overline{D}_{l'}\stackrel{r'}{\rightarrow} \bP^1\]
to 
	\[\widetilde{C}\stackrel{\pi}{\rightarrow} C\stackrel{p}{\rightarrow}\bP^1.\]
In particular, $b_\tau:D_\tau\rightarrow \overline{D}_l$ and $b'_\tau:D'_\tau\rightarrow \overline{D}_{l'}$ are isomorphic.
\end{enumerate}
\end{proposition}
\begin{proof}
	The first claim is the content of \cite[Ex.~2.15(4)]{MR1188194} which we now recall using the notation therein. To apply the tetragonal construction to $\widetilde{D}_l\stackrel{\pi_l}{\rightarrow} D_l\stackrel{q}{\rightarrow}\bP^1$, we consider the following commutative diagram
\begin{equation*}
\begin{tikzcd}
	q_*\widetilde{D}_l\arrow[hookrightarrow]{r}{} \arrow{d}{16:1} & \widetilde{D}_l^{(4)}\arrow{d}{\pi_l^{(4)}}\\
	\bP^1\arrow[hookrightarrow]{r}{} &D_l^{(4)}
\end{tikzcd}
\end{equation*} 
where the superscript $^{(n)}$ denotes the $n$-th symmetric product, the bottom horizontal arrow is defined by sending a point $y\in\bP^1$ to the fiber $q^{-1}(y)$, and $q_*\widetilde{D}_l$ is the fiber product which is a curve in $\widetilde{D}_l^{(4)}$. As discussed in \cite[\S2.1]{MR1188194}, the covering involution associated with $\pi_l:\widetilde{D}_l\rightarrow D_l$ induces an involution $\iota_l$ on $q_*\widetilde{D}_l$. Denote the orientation double cover of $\widetilde{D}_l\stackrel{\pi_l}{\rightarrow} D_l\stackrel{q}{\rightarrow}\bP^1$ by $\widetilde{\bP}^1\rightarrow \bP^1$ (cf.~ \cite[\S2.2]{MR1188194}). Then the map $q_*\widetilde{D}_l\stackrel{16:1}{\rightarrow} \bP^1$ factors as 
	\[q_*\widetilde{D}_l\stackrel{2:1}{\longrightarrow} q_*\widetilde{D}_l/\iota_l\stackrel{4:1}{\longrightarrow} \widetilde{\bP}^1\stackrel{2:1}{\longrightarrow} \bP^1.\] 
Moreover, the orientation double cover $\widetilde{\bP}^1\rightarrow \bP^1$ is trivial; the curves $q_*\widetilde{D}_l$, $q_*\widetilde{D}_l/\iota_l$ and $\widetilde{\bP}^1$ are thus reducible and we obtain the other two towers associated with the tower $\widetilde{D}_l\stackrel{\pi_l}{\rightarrow} D_l\stackrel{q}{\rightarrow}\bP^1$ via the tetragonal construction (see \cite[\S2.5]{MR1188194} for more details). 

	Going back to our case, we need to construct injections $\widetilde{C}\cup \widetilde{Q}\hookrightarrow q_*\widetilde{D}_l$ and $\widetilde{D}_{l'}\hookrightarrow q_*\widetilde{D}_l$. Geometrically, there exists a plane $P\subset \bP^4$ with $X\cap P=L\cup l\cup l'$ as $l$ and $l'$ are contained in the same fiber of $\pi_L:\mathrm{Bl}_LX\rightarrow \Pi$. Consider the conic fibration $\pi_l:\mathrm{Bl}_lX\rightarrow \bP^2_l$. The plane $P$ corresponds to a point $d_{L\cup l'}\in D_l\subset \bP^2_l$. Projecting the discriminant quintic $D_l$ from $d_{L\cup l'}$, one gets the degree $4$ map $q:D_l\rightarrow \bP^1$. Let us now fix a general point $y\in\bP^1 \subset \bP^2_l$. Pulling back the line $\langle d_{L\cup l'},y \rangle\subset \bP^2_l$ joining $d_{L\cup l'}$ and $y$ via $\pi_l:\mathrm{Bl}_lX\rightarrow \bP^2_l$, we obtain a hyperplane $H_y\subset \bP^4$ which intersects $X$ along a smooth cubic surface $X\cap H_y$. Note that the lines $L$, $l$ and $l'$ are contained in $X\cap H_y$. From the configuration of the $27$ lines on a smooth cubic surface, we deduce that besides $L\cup l'$ there are four other pairs of coplanar lines $m_y^{(i)}\cup n_y^{(i)}$ (for $1\leq i\leq 4$) on $X\cap H_y$ meeting $l$ which corresponds to the fiber $q^{-1}(y)\in D_l^{(4)}$. Now we choose a line on $X\cap H_y$ meeting $L$ which is different from $l$ or $l'$; such a line is parametrized by a point of $\widetilde{C}\cup \widetilde{Q}$. Observe that this line intersects four of the eight lines $m_y^{(1)}, n_y^{(1)},\dots, m_y^{(4)}, n_y^{(4)}$, one in each of the four coplanar pairs, and hence defines an element of $q_*\widetilde{D}_l$. Letting $y$ vary and using continuity, we obtain a map $\widetilde{C}\cup \widetilde{Q}\rightarrow q_*\widetilde{D}_l$ which is clearly injective. The definition of $\widetilde{D}_{l'}\hookrightarrow q_*\widetilde{D}_l$ is similar. (Note also that the local pictures of this tetragonal construction are given in \cite[2.14(3)]{MR1188194}.)
	
	Let us now move to the proof of the second assertion. Following \cite[\S2.3]{MR1188194}, we recall that the bigonal construction associates with the tower of double covers $D_\tau\stackrel{b_\tau}{\rightarrow} \overline{D}_l\stackrel{r}{\rightarrow} \bP^1$ another tower of double covers $r_*D_\tau\rightarrow r_*D_\tau/\iota_\tau\rightarrow \bP^1$, where $r_*D_\tau$ is defined via the following fiber product diagram (where the bottom horizontal arrow is defined by sending $y\in\bP^1$ to $r^{-1}(y)\in \overline{D}_l^{(2)}$) 
\begin{equation*}
\begin{tikzcd}
	r_*D_\tau\arrow[hookrightarrow]{r}{} \arrow{d}{4:1} & D_\tau^{(2)}\arrow{d}{b_\tau^{(2)}}\\
	\bP^1\arrow[hookrightarrow]{r}{} &\overline{D}_l^{(2)}
\end{tikzcd}
\end{equation*} 
and $\iota_\tau$ denotes the involution on $r_*D_\tau$ induced by the covering involution of $b_\tau:D_\tau\rightarrow \overline{D}_l$. The proof of the second claim is similar to that of the first one, but we need to take the non-Eckardt type involution into consideration. Specifically, we project $X$ from the $\tau$-invariant line $l$ to $\bP^2_l$ which is also $\tau$-invariant and admits an involution $\tau_{\bP^2_l}$. Note that the point $d_{L\cup l'}\subset \bP^2_l$ corresponding to the plane $P=\langle L,l,l'\rangle$ is fixed by $\tau_{\bP^2_l}$. Note also that the discriminant double cover $\pi_l:\widetilde{D}_l\rightarrow D_l$ is equivariant with respect to the involutions $\tau$ on $\widetilde{D}_l$ and $\tau_{D_l}$ on $D_l$ (cf.~Proposition \ref{tauDl}); the double cover $b_\tau:D_\tau\rightarrow \overline{D}_l$ is obtained as the quotient of $\pi_l$ by the involutions. Now let us project $D_l$ from $d_{L\cup l'}$ to an invariant complementary line $\bP^1 \subset \bP^2_l$ and fix a general point $y\in\bP^1$. Again let $H_y\subset \bP^4$ be the $\tau$-invariant hyperplane corresponding to the line $\langle d_{L\cup l'},y \rangle\subset \bP^2_l$. The key observation is that the smooth cubic surface $X\cap H_y$ admits an involution $\tau_y$ whose fix locus consists of a line and three distinct points (e.g.~\cite[\S9.5.1]{MR2964027}). To verify this, we choose coordinates as in \S\ref{projX2}, noting that $d_{L\cup l'}=[0,0,1]\in \bP^2_l\cong \bP^2_{x_0,x_1,z}$. We then suppose that $y=[a_0,a_1,0]$ and plug $a_0x_1=a_1x_0$ into Equation \eqref{rewrite eqn}. As in the proof of the first assertion, let $m_y^{(i)}\cup n_y^{(i)}$ (with $1\leq i\leq 4$) denote the four coplanar pairs of  lines on $X\cap H_y$ meeting $l$ which are different from $L\cup l'$. It is not difficult to see that these four pairs of lines are related by the involution $\tau_{D_l}$ on $D_l$ and hence give an element of $\overline{D}_l^{(2)}$. Without loss of generality, we assume that $\tau(m_y^{(1)})=m_y^{(3)}$ (respectively, $\tau(n_y^{(1)})=n_y^{(3)}$) and $\tau(m_y^{(2)})=m_y^{(4)}$ (respectively, $\tau(n_y^{(2)})=n_y^{(4)}$). Now choose a $\tau_y$-invariant line on $X\cap H_y$ meeting the pointwise fixed line $L$; there are four such lines besides $l$ and $l'$ all of which are parameterized by points on $\widetilde{C}$. This $\tau_y$-invariant line intersects two of the four lines $m_y^{(1)}, n_y^{(1)}, m_y^{(2)}, n_y^{(2)}$, one in each of the two coplanar pairs. Similarly as in \cite[Ex.~2.15(4)]{MR1188194}, we get an element in $r_*D_\tau$ and thus define an injection $\widetilde{C}\hookrightarrow r_*D_\tau$. Since $(X,\tau,l)$ is general, $r_*D_\tau$ is smooth and irreducible (i.e.~ the situation in \cite[p.~69 (v)]{MR1188194} does not happen) and therefore $\widetilde{C}= r_*D_\tau$. Summing it up, $D_\tau\stackrel{b_\tau}{\rightarrow} \overline{D}_l\stackrel{r}{\rightarrow} \bP^1$ and $\widetilde{C}\stackrel{\pi}{\rightarrow} C\stackrel{p}{\rightarrow}\bP^1$ are related by the bigonal construction.  

	Similarly, $D'_\tau\stackrel{b'_\tau}{\rightarrow} \overline{D}_{l'}\stackrel{r'}{\rightarrow} \bP^1$ and $\widetilde{C}\stackrel{\pi}{\rightarrow} C\stackrel{p}{\rightarrow}\bP^1$ are also related by the bigonal construction, noting that $l'$ is contained in the same fiber of $\pi_L:\mathrm{Bl}_LX\rightarrow \Pi$ as $l$ and that the bigonal structure $p=p_{l\cup l'}:C\rightarrow \bP^1$ on $C$ is determined by $l\cup l'$. Since the bigonal construction is symmetric (cf.~\cite[Lem.~2.7]{MR1188194}), $b_\tau:D_\tau\rightarrow \overline{D}_l$ and $b'_\tau:D'_\tau\rightarrow \overline{D}_{l'}$ are isomorphic.
\end{proof}

	As a consequence of \cite[\S3]{MR817884} and Proposition \ref{bigonalLl}, the Prym varieties $P(\widetilde{C},C)$ and $P(D_\tau,\overline{D}_l)$ are dual to each other; this matches our results in Theorems \ref{inv/antiinv part1} and \ref{inv/antiinv part2}. We conclude the discussion using the following proposition.
\begin{proposition} \label{genericfibC}
	Consider the Prym map $\calP_{2,4}:\calR_{2,4}\rightarrow \calA_3^{(1,2,2)}$. A general member $A\in \calA_3^{(1,2,2)}$ can be realized as the invariant part $JX^\tau$ of the intermediate Jacobian of a general cubic threefold $(X,\tau)$ with a non-Eckardt type involution. Furthermore, set $C\subset X$ to be the pointwise fixed plane section as in Lemma \ref{Csmooth}; note that $C$ is smooth and of genus $1$. Then the generic fiber of $\calP_{2,4}:\calR_{2,4}\rightarrow \calA_3^{(1,2,2)}$ over $JX^\tau\in \calA_3^{(1,2,2)}$ is birational to $C$. 
\end{proposition}
\begin{proof}
The first assertion is a corollary of Theorem \ref{global torelli} and Proposition \ref{inftorelli}. The second claim follows from Proposition \ref{bigonalLl} and the argument in the proof of \cite[Thm.~1.2]{NO_torelli} or \cite[Thm.~5.2]{MR4435960}. Specifically, notation remains the same as in \S\ref{projX1} and \S\ref{projX2}. From Theorem \ref{inv/antiinv part2}, we know that the invariant part $JX^\tau$ is isomorphic to the Prym variety $P(D_\tau,\overline{D}_l)$ where $l\subset X$ is a general invariant line and $b_\tau:D_\tau\rightarrow \overline{D}_l$ is the quotient discriminant double cover (see Proposition \ref{kleincover}). As argued in the proof of \cite[Thm.~1.2]{NO_torelli} or \cite[Thm.~5.2]{MR4435960}, the generic fiber of the Prym map $\calP_{2,4}:\calR_{2,4}\rightarrow \calA_3^{(1,2,2)}$ over $P(D_\tau,\overline{D}_l)$ is birational to $\Pic^2(E)\cong E$ where $E$ is an elliptic curve obtained via the bigonal construction for $D_\tau\stackrel{b_\tau}{\rightarrow} \overline{D}_l\rightarrow \bP^1$. By Proposition \ref{bigonalLl}, we have that $E\cong C$ which completes the proof. 
\end{proof}

\bibliography{cubicthreefold}
\end{document}